\documentclass{article}
\usepackage{graphicx} %
\usepackage{amsmath,amssymb,amsthm}
\usepackage{algorithm}
\usepackage{algorithmic}
\usepackage{multicol}
\usepackage{fullpage}
\usepackage{hyperref}
\usepackage{xcolor}
\usepackage{mleftright}
\usepackage{enumitem}
\usepackage{caption}
\usepackage{float}
\usepackage{hhline}
\usepackage{tikz}
\usetikzlibrary{arrows.meta}
\usepackage{tikz-cd}
\usepackage{quiver}
\usepackage[nameinlink,capitalize]{cleveref} 

\newcommand{\qand}{\quad\And\quad}
\newcommand{\qwhere}{\quad\text{where}\quad}

\newcommand{\B}{{\mathcal{B}}}

\newcommand{\C}{{\mathbb{C}}}
\newcommand{\R}{{\mathbb{R}}}

\newcommand{\N}{{\mathbb{N}}}

\newcommand{\eps}{\varepsilon}

\newcommand{\inr}[2]{\left\langle#1,#2\right\rangle}

\renewcommand{\Re}{\mathrm{Re}}

\newcommand{\wt}{\widetilde}

\newcommand{\abs}[1]{\mleft|#1\mright|}
\newcommand{\rabs}[1]{|#1|}
\newcommand{\magn}[1]{\left\|#1\right\|}
\newcommand{\rmagn}[1]{\|#1\|}

\newcommand{\pare}[1]{\mleft(#1\mright)}

\newcommand{\set}[1]{{\left\{{#1}\right\}}}
\newcommand{\bmat}[1]{\begin{bmatrix}#1\end{bmatrix}}

\newcommand{\spliteq}[2]{\begin{equation}#1\begin{split}#2\end{split}\end{equation}}
\newcommand{\eq}[1]{\begin{equation}{#1}\end{equation}}

\DeclareMathOperator*{\E}{\mathbb{E}}

\DeclareMathOperator*{\argmin}{arg\,min}
\DeclareMathOperator*{\argmax}{arg\,max}

\DeclareMathOperator*{\col}{Col}

\DeclareMathOperator*{\spn}{span}
\DeclareMathOperator*{\rank}{rank}
\DeclareMathOperator*{\tr}{tr}

\DeclareMathOperator{\ball}{Ball}

\DeclareMathOperator{\diag}{diag}

\DeclareMathOperator{\dist}{dist}

\newtheorem{theorem}{Theorem}[section]
\newtheorem{lemma}[theorem]{Lemma}

\newtheorem{corollary}[theorem]{Corollary}

\makeatletter
\newtheorem*{rep@theorem}{\rep@title}
\newcommand{\newreptheorem}[2]{%
\newenvironment{rep#1}[1]{%
 \def\rep@title{#2 \ref{##1}}%
 \begin{rep@theorem}}%
 {\end{rep@theorem}}}
\makeatother

\newreptheorem{theorem}{Theorem}
\newreptheorem{coro}{Corollary}

\theoremstyle{definition}

\newtheorem{remark}{Remark}

\crefname{equation}{}{}

\title{Matrix Factorizations with Uniformly Random Pivoting}

\author{Isabel Detherage \\ UC Berkeley \and Rikhav Shah\footnote{Supported by NSF CCF-2420130} \\ UC Berkeley}
\date{\today}

\DeclareMathOperator{\off}{off}

\newcommand{\tot}{^{(t)}}
\newcommand{\totp}[1]{^{(t+#1)}}
\newcommand{\totm}[1]{^{(t-#1)}}
\newcommand{\totpm}[1]{^{(t\pm #1)}}
\newcommand{\tots}{^{(t)*}}

\newcommand{\tos}{^{(s)}}

\newcommand{\factor}{\pare{1-\frac{k(k-1)}{n(n-1)}}}
\newcommand{\twofactor}{\pare{1-\frac2{n(n-1)}}}

\usepackage{blkarray}

\newcommand{\tacc}{T_{\textnormal{acc}}}

\begin{document}

\maketitle

\begin{abstract}
This paper highlights a formal connection between two families of widely used matrix factorization algorithms in numerical linear algebra. One family consists of the Jacobi eigenvalue algorithm and its variants for computing the Hermitian eigendecomposition and singular value decomposition. The other consists of Gaussian elimination and the Gram-Schmidt procedure with various pivoting rules for computing the Cholesky decomposition and QR decomposition respectively.

Both families are cast as special cases of a more general class of factorization algorithms. 
We provide a randomized pivoting rule that applies to this general class (which differs substantially from the usual pivoting rules for Gaussian elimination / Gram-Schmidt) which admits a unified analysis of the entire class of algorithms. The result is the same linear rate of convergence for each algorithm, irrespective of which factorization it computes.

One important consequence of this randomized pivoting rule is a provable polynomial bound on the numerical stability of the Jacobi eigenvalue algorithm without any preconditioning, which addresses a longstanding open problem of Demmel and Veseli\'c `92 \cite{b0}.

\end{abstract}

\section{Introduction}
At first glance, the Jacobi eigenvalue algorithm for computing the eigendecomposition of Hermitian matrices and the Gram-Schmidt procedure for computing the QR decomposition of tall matrices are quite distinct: they take different inputs, compute different factorizations, have different asymptotic complexities, etc. However, there is an important commonality: both methods are iterative and consider only pairs of columns of their inputs at a time. In this work, this connection is taken further, and both algorithms are realized as special cases of a more general matrix factorization algorithm.

Our primary result, \Cref{a0}, is a unified analysis of all special cases of this general procedure when using a particular randomized pivoting rule; we show that these algorithms are all linearly convergent in expectation, and further that the rate of convergence is unaffected by which particular factorization is being computed.

In finite arithmetic, we provide a framework for showing that instances of this general procedure are numerically stable. This framework is applied to several orthogonalization methods (\Cref{a1}), as well as the Jacobi eigenvalue algorithm (\Cref{a2}, \Cref{a3}, \Cref{a4}), building on the work of \cite{b0} in the latter case.

We first give a brief overview of the history and context of these algorithms. As these algorithms have been extensively studied, we cannot give a comprehensive overview. We instead highlight their origins, practical relevance today, and interactions between them.

\medskip
\noindent
\textit{Eigendecomposition / Singular value decomposition:}
The classical Jacobi eigenvalue algorithm (which we henceforth refer to as the two-sided Jacobi eigenvalue algorithm\footnote{
The literature also uses ``two-sided'' to refer to a version of the Jacobi eigenvalue algorithm due to \cite{b1} for computing the singular value decomposition; that method is different from the version of the Jacobi eigenvalue for computing the singular value decomposition that this paper considers.}) is an iterative algorithm for computing an approximate eigendecomposition of a given Hermitian\footnote{One can replace $\C$ with $\R$ throughout the paper. Hermitian becomes symmetric, (Special) unitary becomes (special) orthogonal, and the Hermitian adjoint $(\cdot)^*$ becomes the transpose $(\cdot)^\top$.} matrix that dates back to 1846 \cite{b2}.
It was rediscovered and popularized in the 1940s and 50s when the advent of computers made large numerical computations feasible (we refer the reader to \cite{b3} for a detailed history).
Today, it is included in popular textbooks \cite{b4,b5} and software packages \cite{b6}.
Although it is often less preferred due to a slower runtime \cite{b4}, it has several strengths when compared with other popular methods for diagonalization. First, if the given matrix $B$ is already close to diagonal, the arithmetic cost is greatly reduced \cite{b7}. Second, the computation is highly parallelizable, and each thread only needs a small amount of memory (see, for example, \cite{b8}). Third, numerical experiments reveal classes of matrices for which all the eigenvalues appear to be computed with small relative error \cite{b0}.

The corresponding algorithm for computing the singular value decomposition of a given matrix $A\in\C^{d\times n}$, today called the one-sided Jacobi eigenvalue algorithm, was introduced as an independent algorithm by Hestenes \cite{b9}, then rediscovered and connected to the classical two-sided Jacobi eigenvalue algorithm by Chartres \cite{b10}. It can be viewed as
implicitly running the two-sided Jacobi eigenvalue algorithm on the matrix $B=A^*A$, without ever forming $B$. It enjoys many of the same advantages as the two-sided Jacobi eigenvalue algorithm \cite{b11,b0}.

The numerical stability of both the one- and two-sided variants was investigated by Demmel and Veseli{\'c} in a highly influential paper \cite{b0}.
They primarily focus on the accuracy with which the small eigenvalues of positive definite matrices are computed. They find ``modulo an assumption based on extensive numerical tests'' that the Jacobi eigenvalue algorithm essentially produces the best approximation of the small eigenvalues one could hope for in the floating point arithmetic model. A family of matrices with more problematic behavior that those tested by \cite{b0} was published by \cite{b12}, though they do not result in instability of the method overall.

\medskip
\noindent
\textit{Cholesky decomposition / QR decomposition: } Gaussian elimination is the oldest method for factoring a given matrix $B$ as $B=LU$ where $L$ is lower triangular and $U$ is upper triangular \cite{b13}. When $B$ is positive definite, one is guaranteed that $U=L^*$ resulting in the Cholesky decomposition $B=LL^*$. Gaussian elimination with various pivoting rules is the primary algorithm used in practice for the Cholesky decomposition \cite{b14,b15,b6}.

There are many methods for the QR decomposition. 
Both classical and modified Gram-Schmidt applied to $A\in\C^{d\times n}$ are mathematically equivalent to Gaussian elimination applied to the matrix $B=A^*A$ \cite{b16}.
Some other algorithms actively exploit the connection between the QR and Cholesky decompositions: \cite{b17} proposes computing the Cholesky decomposition $A^*A=R^*R$ and then computing $Q=AR^{-1}$, and the algorithm(s) of \cite{b18} and \cite{b19} explore this idea further, with modifications to improve stability.

\medskip
There is a notable way to use the Cholesky decomposition and one-sided Jacobi eigenvalue in conjunction for the purpose of a symmetric eigendecomposition as well. The idea is to compute the Cholesky factorization $B=LL^*$ of an input, and compute the SVD of $L^*=U\Sigma V^*$, which then implies the eigendecomposition $B=V\Sigma^2V^*$. The proposal of \cite{b20} is to compute this SVD by applying the one-sided Jacobi eigenvalue algorithm to $L$ (not to $L^*$). This technique avoids the aforementioned assumption needed by \cite{b0} in their theoretical analysis of the numerical stability of the Jacobi eigenvalue algorithm \cite{b21}. However, while it ensures accurate eigenvalues, it does not ensure accurate eigenvectors \cite{b22}. Nevertheless, it is observed to be faster \cite{b0}. The version of this technique which applies one-sided Jacobi to $L^*$ is also faster, though it lacks a theoretical or practical advantage in terms of numerical stability \cite{b0}. Several additional variants of this idea were studied by \cite{b23}.

Several works take inspiration from algorithms for one task for the other. For instance, \cite{b24} proposes a Jacobi-like method for the QR decomposition, applying Givens rotations to transform a matrix $A$ into an upper triangular one $R$. The accumulation of all the Givens rotations then forms $Q$ (which is akin to the method for the QR decomposition using Householder reflections \cite{b14}). This method for QR now commonly appears in numerical linear algebra textbooks, e.g. \cite{b4}.

\medskip
The central focus of this paper is the striking similarity between the operations performed by the one-sided Jacobi eigenvalue algorithm and the modified Gram-Schmidt procedure. 
This similarity has been previously noted, for instance \cite{b25} observes that the operations are actually {the same} under a certain kind of infinite matrix limit.
We show that this similarity can be exploited to obtain a unified analysis of these algorithms and more.

\subsection{Paper overview}
In \Cref{a5}, we present the Jacobi eigenvalue algorithm and modified Gram-Schmidt procedure (MGS) in a way that highlights their similarity. In \Cref{a6} we formally describe an abstract algorithm which encompasses the Jacobi eigenvalue algorithm and MGS as special cases.

Our results about this generalized algorithm are presented in the remaining sections. In \Cref{a7} we introduce our main tool for evaluating progress in the algorithm, a certain potential function, and we relate this potential function to other standard metrics. \Cref{a8} focuses on exact-arithmetic implementations of the algorithm, in the real RAM model. For inputs are not too ill-conditioned, we show that $O(n^2\log(n/\delta))$ iterations of this procedure suffice to produce $\delta$-approximate factorizations (stated precisely in \Cref{a0}). For many reasonable special cases of this generalized algorithm, each iteration costs only $O(n)$ arithmetic operations, resulting in many $\wt O(n^3)$-time factorization algorithms. Additionally, this section confirms a recent conjecture of Steinerberger about the convergence of particular method for orthogonalization \cite{b26} (see \Cref{a9}).

\Cref{a10} studies the general algorithm in finite arithmetic. Our focus is on proving statements roughly of the form ``when machine precision is at most $n^{-c}$ (for an explicit, small $c$), then the results of the algorithm have a guaranteed level of accuracy.''
Our analysis does not optimize $c$, and we do not advertise bounds that are particularly attractive for large matrices in single and double floating point arithmetic. Nevertheless, the dependence on $n$ on the bits of precision required is shown to be logarithmic, and we prove that quad precision can handle a substantial range of inputs.

We have two primary results about finite arithmetic: the first is that a large class of orthogonalization algorithms are numerically stable (including an algorithm for computing the QR decomposition specifically).
The second addresses the ``assumption based on extensive numerical tests'' needed by Demmel and Veseli{\'c} in \cite{b0} to show stability of the Jacobi eigenvalue algorithm. This assumption is that a certain condition number remains bounded as iterations are performed. Their experiments showed that it doesn't seem to rise more than a constant amount above its initial value, but their proofs only bound it by certain exponentially growing quantities (exponential either in the size of the input or number of iterations performed). We show that with uniformly random pivoting, it is bounded by a polynomial in the input size and number of iterations performed (see \Cref{a11}).

\begin{center}
    \begin{tabular}{|l|l|}
        \hline
        \textbf{Notation} & \textbf{Meaning} \\
        \hline
        $\magn{\cdot}$ & $\ell_2$ vector or operator norm\\
        $\magn{\cdot}_F$ & Frobenius norm\\
        $\lambda_j(A)$ & $j^{\textnormal{th}}$ largest eigenvalue of $A$\\
        $\kappa(A)=\magn A\cdot\rmagn{A^{-1}}$               & Condition number of $A$\\
        $D_A = \diag(A_{11}, A_{22}, \hdots, A_{nn})$          & Diagonal part of $A$\\
        $\hat A=D_A^{-1/2}AD_A^{-1/2}$ & Diagonal-normalization of $A$\\
        $\hat\kappa(\cdot)$&Diagonal normalized condition number, $\hat\kappa(A):=\kappa(\hat A)$\\
        $A\odot B$ & Hadamard or entrywise product, $(A\odot B)_{ij}=a_{ij}b_{ij}$\\
        $O(\cdot),\wt O(\cdot)$ & $O(\cdot)$ hides constants, $\wt O(\cdot)$ hides polylogarithmic factors.\\
        $[n]=\set{1,\ldots,n}$ & Set of indices 1 through $n$\\
        \hline
        \end{tabular}
\end{center}

\section{Connecting SVD and QR}\label{a5}
\subsection{The core connection}
Here we describe a simple connection between the SVD and QR decomposition (and the corresponding connection between the eigendecomposition and Cholesky decomposition) and algorithms for computing them.
To see the connection, we first express the QR decomposition and a factorization that is almost the SVD in a common form: given $A$ with full column rank, write
\eq{\label{a12}A=QT^{-1}}
where $Q$ has orthogonal columns and $T$ is nonsingular. When $T$ is unitary, by further decomposing $Q=U\sqrt{Q^*Q}$, one obtains the singular value decomposition. When $T$ is upper triangular and $Q^*Q=I$, this is exactly the QR decomposition. Analogously, the eigendecomposition and Cholesky decomposition of $B=A^*A$ can both be expressed as $B=(T^{-1})^*DT$ where $D=Q^*Q$. The four cases are summarized in the following table.

\begin{table}[H]\centering
\begin{tabular}{|l|c|c|}\hline
               & $B=(T^{-1})^*DT^{-1}$            & $A=QT^{-1}$ \\\hline
\begin{tabular}[c]{@{}l@{}}$T$ is unitary and\\$D=Q^*Q$ is diagonal\end{tabular} & Eigendecomposition & Singular value decomposition  \\\hline
\begin{tabular}[c]{@{}l@{}}$T$ is upper triangular and\\$D=Q^*Q=I$\end{tabular}
               & Cholesky decomposition & QR decomposition   \\\hline\end{tabular}
\end{table}

\noindent
Notably, both Modified Gram-Schmidt (MGS) and one-sided Jacobi can be interpreted as methods for producing an orthogonal basis $Q=\bmat{q_1&\cdots&q_n}$ for a given collection of vectors, expressed as columns of $A=\bmat{a_1&\cdots&a_n}$ (along with the corresponding column operation sending $A$ to $Q$). The algorithms themselves closely resemble each other too: they both iteratively pick pairs of indices $(i,j)$ and replace the columns $a_i,a_j$ with a different orthogonal basis for their span. By repeating this many times with appropriate selection of $(i,j)$ each time, the collection of vectors will approach an orthogonal set.
More specifically let $A\tot=\bmat{a_1\tot&\cdots&a_n\tot}$ be the state after this has been done for $t$ pairs $(i,j)$. For each $t$, the Jacobi eigenvalue algorithm selects a pair $(i,j)$ and then updates
\spliteq{\label{a13}}{
    a_i\totp1&=\cos(\theta)a_i\tot + \sin(\theta)a_j\tot
    \\
    a_j\totp1&=-\sin(\theta)a_i\tot + \cos(\theta)a_j\tot
    \\
    a_\ell\totp1&=a_\ell\tot \quad\forall\ell\not\in\set{i,j}}
where $\theta$ is such that the resulting columns $a_i\totp1$ and $a_j\totp1$ are orthogonal. MGS selects $i<j$ then updates
\spliteq{\label{a14}}{
a_{j}\totp1&=\frac{a\tot_{j}-\inr{a\tot_{j}}{a\tot_{i}}a\tot_{i}}{\magn{a\tot_{j}-\inr{a\tot_{j}}{a\tot_{i}}a\tot_{i}}}
\\
a_\ell\totp1&=a_\ell\tot \quad\forall\ell\neq j.}
The update equation \eqref{a13} is a Givens rotation of $A\tot$ (also called a plane rotation); importantly it is unitary. The product of these rotations forms $T$ in the decomposition \eqref{a12}.
The update equation \eqref{a14} is an elementary matrix operation followed by a scaling; importantly, it is upper triangular. The product of these elementary operations forms $T$ in the decomposition \eqref{a12}.

In both cases, the result of this update is that the $i$th and $j$th columns of $A\totp1$ will be orthogonal. Despite their similarity, the existing analysis of these algorithms is quite different, as is their empirical performance measured in terms of computational and communication costs and numerical stability.

\subsection{Pivoting rules}
In both cases, the strategy for picking $(i,j)$ at each time $t$ is determined by a ``pivoting rule''.
Multiple pivoting rules have been proposed; we group them into three categories to discuss: fixed, adaptive, and randomized. Before doing so, we must clarify some terminology.
There is some subtlety in what the literature refers to as ``one iteration'' of either procedure. For the Jacobi eigenvalue algorithm, one iteration is sometimes taken to mean one ``sweep'' (which we describe in a moment) of all $n(n-1)/2$ pairs $(i,j)$. For MGS (more specifically, left-looking MGS), 
one iteration typically refers to the grouping of updates
\eq{\label{a15}
(1,i),(2,i),\ldots,(i-1,i)}
for each $i$. For Gaussian elimination (and right-looking MGS), it typically refers to the grouping of updates
\eq{\label{a16}
(i,i+1),(i,i+2),\ldots,(i,n)}
for each $i$. Because of this, the word ``pivot'' in those contexts typically refer to the single column index $i$. In this paper, we refer to each selection of $(i,j)$ and update of the corresponding column(s) as one iteration, i.e. $A\tot$ defined by either \eqref{a13}, \eqref{a14} is the state after $t$ iterations. One pivot then corresponds to a pair of column indices $(i,j)$.

\textit{Fixed pivot rules: }
These are rules that fix a sequence of $(i,j)$ ahead of time.
The ones we highlight for the Jacobi eigenvalue algorithm are called cyclic or sweep rules, due to \cite{b27}, see also \cite{b28}.
The row-cyclic version uses the sequence
\[(1,2),(1,3),(1,4),(1,5),\ldots,(1,n),(2,3),(2,4),(2,5),\ldots,(2,n),\ldots (n-1,n)\]
and repeats it over and over again. The column-cyclic version uses the sequence
\[(1,2),(1,3),(2,3),(1,4),(2,4),(3,4),(1,5),(2,5),(3,5),(4,5),\ldots,(n-1,n)\]
and repeats it over and over again. Each run through the sequence of $n(n-1)/2$ pairs is referred to as one ``sweep.''
In the context of Gaussian elimination and MGS, those sweeps correspond to applying the groupings \eqref{a15} or \eqref{a16} for $i=1,\ldots,n-1$. Running MGS for just one or two sweeps are both popular proposals.

For appropriately chosen rotations, \cite{b28} show that row-cyclic and column-cyclic Jacobi converges, and \cite{b7} improve this to (eventual) quadratic convergence. However, there is no rigorous theory to predict the number of sweeps required to achieve a desired distance to diagonal \cite{b17}. In practice, it is observed to be a constant number of sweeps \cite{b4}. Better theoretical results are available for the next category of pivoting rules.

\textit{Adaptive pivot rules:} These are pivoting rules that use the current value of $A\tot$ to decide the next pivots. The original rule proposed by Jacobi \cite{b2}, which we call the greedy pivoting rule, was to pick
\eq{\label{a17}(i,j)=\argmax_{i,j:i\neq j}\abs{\inr{a_i}{a_j}}.}
This is somewhat similar to column-pivoting rule of Gram-Schmidt (implicitly equivalent to complete pivoting rule of Gaussian elimination). This rule for Gram-Schmidt maintains a list of indices $\mathcal I$ initialized to $[n]$. Then
\eq{\label{a18}i=\argmax_{i\in\mathcal I}\magn{a_i}} is computed, $i$ is removed from $\mathcal I$, and the updates $\set{(i,j):j\in\mathcal I}$ are performed (in any order). This is repeated until $\mathcal I$ is empty.
Strictly speaking, these rules will not necessarily compute the QR decomposition of $A$, since we no longer are assured that the update pairs $(i,j)$ satisfy $i<j$. However, if one lets $\sigma$ be the permutation which records the order in which indices are removed from $\mathcal I$, then we are assured $\sigma(i)<\sigma(j)$ and therefore the method computes QR decomposition of $AP_\sigma$ where $P_\sigma$ is the permutation matrix associated with $\sigma$.

The same pivoting strategy has been considered for one-sided Jacobi \cite{b29}. Numerical experiments showed that this pivoting strategy can improve the rate of convergence in several (but not all) cases (see \cite{b29} for more details).

\textit{Randomized pivot rules:}
There seems to be much less attention given to randomized pivot rules. A recent work of Chen et al. \cite{b30} considers the partial Cholesky and QR decompositions for computing low-rank approximations. They study a randomized variant of the complete pivoting rule \eqref{a18} where $i\in\mathcal I$ is sampled with probability proportional to $\magn{a_i}^2$. This and similar ideas had been considered before as well \cite{b31,b32,b33}.

One exceedingly simple randomized rule is just to sample $(i,j)$ uniformly at random among all pairs of indices. We were unable to find a reference that explicitly describes this rule (or any randomized rule) for the Jacobi eigenvalue algorithm; one reason that this is surprising is that the proof of global convergence of the greedy pivoting rule \eqref{a17} essentially applies the probabilistic method to the hypothetical algorithm which uses this uniformly random pivot (see, for example, Section 5 of \cite{b34}).

In the context of orthogonalization, we note an important distinction between sampling $(i,j)$ uniformly at random among all pairs of indices and just pairs of indices with $i<j$ (or $\sigma(i)<\sigma(j)$ for some permutation $\sigma$). In the latter case, we are assured the computed decomposition (if the method converges) is the QR decomposition $AP_\sigma=QR$. However, in the former case, we are given no assurances about $R$ being upper triangular. That rule (without the $i<j$ constraint) was proposed by Steinerberger \cite{b26}, and both rules (both with and without the $i<j$ constraint) were concurrently proposed by the present authors in \cite{b35}. 
Both those works were motivated by a designing a preconditioner for the Kaczmarz linear system solver, and made analogies of the method to the Kaczmarz solver itself; neither work identified the connection to the Jacobi eigenvalue algorithm.
Steinerberger conjectured that the rate of convergence of this procedure is $\kappa(A\tot)\approx\exp(-O(t/n^2))\kappa(A)$ based on numerical experiments and a heuristic. The present authors proved that $(-\log\det\rabs{A\tot})\le\exp(-O(t/n^2))(-\log\det\abs A)$.

\subsection{Block variants}
Both Gram-Schmidt and Jacobi admit block variants, which can be abstractly described as follows: rather than pick a \textit{pair} of indices $(i,j)$ at each step and updating columns $i$ and $j$ of $A\tot$, one can instead pick any subset of indices $J\subset[n]$ and update all the corresponding columns. We call $J$ the pivot set. In the non-block case, pivoting on $(i,j)$ ensures that
\[\inr{a\totp1_i}{a\totp1_j}=0.\]
In the block variant, pivoting on $J$ ensures that
\[\inr{a\totp1_i}{a\totp1_j}=0\quad\forall i,j\in J,i\neq j.\]
These methods are often motivated by considering the practical performance of these algorithms on modern hardware. For example, one may be able to take advantage of large cache sizes to perform one iteration for $\abs J>2$ very quickly.

The block Jacobi eigenvalue algorithm of Drma\v c \cite{b36} is one such method. It fixes a partition $I_1\sqcup\cdots\sqcup I_m\subset[n]$ and divides the given matrix $B$ into blocks corresponding to those index sets. Then, the pivot set $J$ is selected to be $I_{\ell}\cup I_{\ell'}$, where the pair $(\ell,\ell')$ can be selected analogously to how $(i,j)$ were selected in the non-block version. Both \cite{b37} and \cite{b38} show convergence of block Jacobi under various pivoting strategies.
There are many block versions of Gram-Schmidt along similar lines, see \cite{b39} for a recent survey.

\subsection{A generalized factorization}\label{a6}

In describing the connection between Gram-Schmidt and the one-sided Jacobi eigenvalue algorithm, we've already hinted at what the generalized algorithm is. Once a pivot set $J$ has been selected, we may abstract away the process by which the corresponding columns are updated: all we assume is that it is an invertible column operation which results in orthogonal columns. We will use the letter $S$ for the $\abs J\times \abs J$ matrix performing this column operation. The Jacobi eigenvalue algorithm corresponds to the special case where $S$ is unitary, and Gram-Schmidt to the case where $S$ is upper triangular.
The general pseudo-code for this and the two-sided version is below. We first clarify some notation. $\text{GL}_k(\C)$ is the set of all invertible $k\times k$ matrices over $\C$. $\binom{[n]}k$ is the set of all subsets of $[n]$ of size $k$. For an index set $J=\set{i_1,\ldots,i_k}\in\binom{[n]}k,i_1<\cdots<i_k$, denote the submatrices
\[A_J=\bmat{
a_{1,i_1} &\cdots& a_{1,i_k}\\
\vdots&&\vdots\\
a_{d,i_1} &\cdots& a_{d,i_k}\\
}\in\C^{d\times k}
\qand
B_{JJ}=\bmat{b_{i_1i_1} &\cdots& b_{i_1i_k} \\ \vdots && \vdots \\ b_{i_ki_1} &\cdots& b_{i_ki_k}}\in\C^{k\times k}.\]
For a permutation $\sigma$ on $[n]$ let $P_\sigma$ be the associated permutation matrix, i.e. $P_\sigma e_{j}=e_{\sigma(j)}$.
\begin{algorithm}[H]
\caption{
This is pseudo-code for the one-sided (left) and two-sided (right) versions of the generalized factorization algorithm. %
}\label{a19}
  \begin{multicols}{2}
\begin{algorithmic}[1]
\REQUIRE $A\in\C^{d\times n}$ has full column rank, $k\ge2$.
\ENSURE $A=Q\tacc$
\STATE $A^{(0)}\gets A$. 
\STATE $\tacc^{(0)}\gets I$.
\STATE $t\gets 0$
\WHILE{$A\tots A\tot$ is not close to diagonal}
\STATE Select a pivot $J=\set{i_1,\ldots,i_k},i_1<\cdots<i_k$.
\STATE Construct $S\in\text{GL}_k(\C)$ such that $A\tot_JS$ has orthogonal columns.
\STATE Define $T\in\C^{n\times n}$ as \(T=P_\sigma^*\bmat{S^{-1}\\&I}P_\sigma\) where $\sigma$ is a permutation such that $\sigma(i_j)=j$.
\STATE $A\totp1\gets A\tot T^{-1}$
\STATE $\tacc\totp1\gets\tacc\tot T$
\STATE $t\gets t+1$
\ENDWHILE
\RETURN $Q\gets A\tot$, $\tacc\gets \tacc\tot$
\end{algorithmic}
\columnbreak
\begin{algorithmic}[1]
\REQUIRE $B\in\C^{n\times n}$ is positive definite, $k\ge2$.
\ENSURE $B=\tacc^*D\tacc$
\STATE $B^{(0)}\gets B$.
\STATE $\tacc^{(0)}\gets I$.
\STATE $t\gets 0$
\WHILE{$B\tot$ is not close to diagonal}
\STATE Select a pivot $J=\set{i_1,\ldots,i_k},i_1<\cdots<i_k$.
\STATE Construct $S\in\text{GL}_k(\C)$ such that $S^*B\tot_JS$ is diagonal.
\STATE Define $T\in\C^{n\times n}$ as \(T=P_\sigma^*\bmat{S^{-1}\\&I}P_\sigma\) where $\sigma$ is a permutation such that $\sigma(i_j)=j$.
\STATE $B\totp1\gets(T^{-1})^*B\tot T^{-1}$
\STATE $\tacc\totp1\gets\tacc\tot T$
\STATE $t\gets t+1$
\ENDWHILE
\RETURN $D\gets B\tot$, $\tacc\gets\tacc\tot$
\end{algorithmic}
\end{multicols}
\end{algorithm}
\noindent We present both versions of the algorithms side-by-side to highlight their mathematical equivalence: one can start with $B=A^*A$ and maintain $A\tots A\tot=B\tot$ by selecting the same $T$ in line 7 of either algorithm. Because of this correspondence, we will refer to various lines of the algorithm without needing specify which version we are talking about. We also highlight that $T$ need not be formed and inverted explicitly.

There are many pivot rules one can use for the selection in line 5. This paper uses a randomized pivoting rule where $J$ is sampled uniformly at random from $\binom{[n]}k$. We refer to this as ``randomized size $k$ pivoting''.

We obtain a unified analysis of all special cases of \Cref{a19} with randomized size $k$ pivoting. The mechanism by which $S$ is selected in line 6 and the additional structure one wishes to impose upon $S$ do not affect our analysis in any way---in expectation, the rate of convergence depends only on $k$ and $n$. For concreteness, \autoref{a20} lists several decompositions one can obtain by imposing different constraints on $S$ and $(A_J\tot S)^*(A_J\tot S)=S^*B_{JJ}\tot S$.

\begin{table}[H]\centering 
\begin{tabular}{|l|c|c|}\hline
Decomposition name                & $S$                   & $(A_J\tot S)^*(A_J\tot S)=S^*B_{JJ}\tot S$ \\\hline
Eigendecomposition / SVD              & Unitary               & Diagonal \\\hline
Cholesky / QR            & Upper triangular      & Identity \\\hline
LDL / unnamed                    & Unit upper triangular & Diagonal \\\hline
unnamed / QL                     & Lower triangular      & Identity \\\hline
Gram matrix / Orthogonalization & Nonsingular           & Identity \\\hline
\end{tabular}
\caption{Two-sided and one-sided decomposition of the form $B=T^*DT$ and $A=QT$ that one obtains by different selections of $S$ in line 6 of \Cref{a19}.}
\label{a20}
\end{table}

\begin{remark}[Using recursion to compute $S$]
Let's revisit line 6 of either algorithm. Computing $S$ is itself a matrix decomposition problem. In fact, the decomposition of $B_{JJ}$ or $A_J$ one must find is exactly the kind of decomposition one seeks to compute of $B$ or $A$. This points to the natural choice of using the algorithm recursively. Alternatively, one may use any expensive factorization algorithm since the problem instance is smaller.
\end{remark}

\begin{remark}[Parallelizability]
When two pivot sets are disjoint, the corresponding updates may occur in parallel. The results in this paper can be extended to the case where multiple disjoint pivots $J_1,\ldots,J_\ell\in\binom{[n]}k$ are sampled at a time where the marginal distribution of each $J_j$ is uniform, and then the corresponding pivots are performed in any order.
\end{remark}
\begin{remark}[Fixed pivot size]
    There's no reason to keep the pivot size $\abs J=k$ constant, other than it simplifies the analysis. $J$ may be sampled from any set such that the probability $\set{i,j}\subset J$ is equal for all distinct pairs $i,j$.
\end{remark}
\begin{remark}[Meaning of ``convergence'']
    The algorithm halts when $A\tots A\tot=B\tot$ is close to diagonal, in whatever sense the user demands. When we say that the method converges, we mean that the criteria that ends the while-loop in the algorithm is met. Other sources may impose a stronger condition on convergence: it means that the version of the algorithm which continues iterating indefinitely, never exiting the while-loop, must produce states such that the limits $\lim_{t\to\infty}A\tot$ (or $\lim_{t\to\infty}B\tot$) and $\lim_{t\to\infty}\tacc\tot$ exist. In particular, if $A(t)$ has reached an orthonormal basis, but that basis changes each iteration, we will still say it has converged.
\end{remark}

\section{Main tool for analysis}\label{a7}
In \Cref{a21} we review the standard Jacobi analysis, from which we draw inspiration. This analysis rests on a certain potential function, used to measure progress of convergence to a diagonal matrix. In \Cref{a22} we introduce our potential function, $\Gamma(\cdot)$, and relate $\Gamma(\cdot)$ to several other metrics.

\subsection{Standard Jacobi analysis}\label{a21}
Our analysis of \Cref{a19} follows a similar strategy as the classical analysis of the Jacobi eigenvalue algorithm. To highlight the connection, we review the classical analysis here.
Say that $B=B^{(0)},B^{(1)},B^{(2)},\ldots$ is the sequence of matrices produced by the iterations of the Jacobi eigenvalue algorithm. The proof strategy is to identify a potential function that behaves monotonically, i.e. some quantity which strictly decreases as iterations are performed. In the classical analysis, this potential is taken to be the squared distance of the current iterate $B\tot$ to the closest diagonal matrix in Frobenius norm, denoted $\off(\cdot)$. This quantity can be compactly expressed as
\[
\off(B)
:=\inf_{D\text{ diagonal}}\magn{B-D}_F^2
=\magn{B-D_B}_F^2
=\sum_{i,j\in[n],i\neq j}\abs{b_{ij}}^2
\]
where $D_B=\diag\pare{b_{11},\ldots,b_{nn}}$.
A straightforward calculation reveals that that\eq{\label{a23}\off(B\totp1)=\off(B\tot)-2\rabs{b\tot_{ij}}^2}
where $(i,j)$ is the pivot selected at time $t$. If $(i,j)$ is picked according to the greedy pivot rule (i.e. $\rabs{b\tot_{ij}}^2=\argmax_{i,j:i\neq j}\rabs{b\tot_{ij}}^2$), then $\rabs{b\tot_{ij}}^2$ will certainly be greater than the average so one has
\[\rabs{b\tot_{ij}}^2\ge\frac1{n(n-1)}\off(B\tot)\]
which ensures global convergence of $\off(B\tot)\to0$ at a linear rate. If instead one uses a randomized size 2 pivoting strategy, one has equality in expected value,
\eq{\label{a24}\E\off(B\totp1)=\twofactor\off(B\tot),}
which again ensures global convergence in expectation at a linear rate. Although both cases reduce to reasoning about the expectation, it does not seem that a random pivoting strategy has been considered in the literature. 

In the more general setting where $B\tot$ is the sequence produced by non-unitary transformations, we no longer have \eqref{a23} (and therefore not \eqref{a24} either).  Instead, we obtain an expression similar to \eqref{a24} but with a different function in place of $\off(\cdot)$.

\subsection{A new potential function}\label{a22}
Our potential function admits a few equivalent expressions. Letting $\hat B = D_B^{-1/2}BD_B^{-1/2}$, i.e., the diagonal-normalization of $B$, we define $\Gamma(B)$ as
\eq{\label{a25}
\Gamma(B)
:=\tr(B\odot B^{-1}-I)
=\tr(\hat B^{-1})-n
=-\sum_{i,j\in[n],i\neq j}b_{ij}\cdot(B^{-1})_{ji}.}
As with $\off(\cdot)$, it is clear from the definition that $\Gamma(B)=0$ when $B$ is diagonal. Less obvious, now, is the converse statement and quantitative versions.
Notably, the converse fails among indefinite or general matrices\footnote{Counter examples include
$\Gamma\pare{\bmat{1&2&1\\2&1&1\\1&1&1}}=0$,
$\Gamma\pare{\bmat{0&1\\0&0}}=0$.}.
Fortunately, among positive definite matrices, $\Gamma(B)$ is proportional to the \textit{Jensen gap} of a randomly sampled eigenvalue of $\hat B$ for the function $\lambda\mapsto\lambda^{-1}$:
\eq{\label{a26}
\frac1n\Gamma(B)=\E(\lambda^{-1})-\pare{\E(\lambda)}^{-1}=\frac1n\tr\pare{\hat B^{-1}}-\pare{\frac1n\tr(\hat B)}^{-1}.}
The expressions \cref{a25}, \cref{a26} are equivalent since $\hat B$ has all 1s along the diagonal, so $\tr(\hat B)=n$. Jensen's inequality directly implies $\Gamma(B)\ge0$, with equality if and only if all the eigenvalues of $\hat B$ are equal, i.e. $\hat B=I$, which is equivalent to $B=D_B$ being diagonal.
Unfortunately, the statement ``$\Gamma(B)=0$ implies $B$ is diagonal'' is not enough to conclude that if for a sequence $B\tot$ with $\Gamma(B\tot)\to0$ that $\off(B\tot)\to0$. Consider, for example,
\[B\tot=\bmat{t^2 & t \\ t & t^2},\quad\hat B\tot=\bmat{1&t^{-1}\\t^{-1}&1}.\]
In this example, $\Gamma(B\tot)=2(t^2-1)^{-1}\to0$ and yet
$\off(B\tot)=2t^2\to\infty$.
This example is capturing the difference between becoming closer to diagonal in an absolute versus relative sense. Indeed $B\tot$ is moving further from the set of diagonal matrices in an absolute sense, but in a relative sense we should instead consider $\off(\hat B\tot)=2t^{-2}\to0$. It is this notion of distance-to-diagonal that is captured by $\Gamma(\cdot)$. Specifically, as shown in \Cref{a27}, $\Gamma(B\tot)\to0$ implies $\off(\hat B\tot)\to0$ at the same rate. We justify considering only this relative sense of distance-to-diagonal by noting that it is used as a standard stopping criterion for the Jacobi eigenvalue algorithm \cite{b0}; therefore, convergence of $\Gamma(B\tot)\to0$ is enough to show convergence of \Cref{a19} for stopping criteria that are satisfied when $\off(\hat B)$ is reduced below some threshold.

This relationship between $\Gamma(\cdot)$ and $\off(\cdot)$ is intuitively explained by a general fact that the Jensen gap is larger for distributions that are ``spread out'' and smaller for distributions that are highly concentrated. In particular, $\off(\hat B)$ is directly proportional to the variance of a randomly selected eigenvalue:
\[\E(\lambda^2)-\E(\lambda)^2=\frac1n\tr(\hat B^2)-1=\frac1n\rmagn{\hat B-I}_F^2=\frac1n\off(\hat B).\]
This same mechanism can be used to control the ratio between the largest and smallest eigenvalues of $\hat B$: if this ratio is close to 1, the eigenvalues are clustered, and if it is large, they are spread out. This ratio is simply the condition number $\kappa(\hat B)$. This quantity plays an important role in the numerical stability analysis of the Jacobi eigenvalue algorithm appearing in \cite{b0}; it is known as the diagonal-normalized condition number, which we denote as $\hat\kappa(B):=\kappa(\hat B)$.
This condition number can be arbitrarily smaller than the usual condition number $\kappa(\cdot)$ (consider, for example, any diagonal matrix), but cannot be much larger (specifically, as shown in \cite{b40}, we have \(\hat\kappa(B)\le n\kappa(B).\)).
The relationship between $\Gamma(\cdot)$ and $\hat\kappa(\cdot)$ presented in \Cref{a28} is what allows us to prove polynomial control over the numerical stability of \Cref{a19}.

These relationships between $\Gamma(B)$ and $\kappa(\hat B)$ and $\off(\hat B)$ are derived by considering an optimization program of the eigenvalues $\lambda_j$ of $\hat B$. The constraints come from the invariant $\tr(\hat B)=n$ and the definition $\Gamma(B)=\Gamma(\hat B)=\tr(\hat B^{-1})-n$. Converting these to statements about the eigenvalues themselves results in the constraints
\spliteq{\label{a29}}{
\lambda_1,\ldots,\lambda_n>0,\quad
\lambda_1+\cdots+\lambda_n=n,\quad\lambda_1^{-1}+\cdots+\lambda_n^{-1}=n+\Gamma(\hat B).}
Strictly speaking, one may also wish to impose the ordering of the eigenvalues, $\lambda_1\ge\cdots\ge\lambda_n$. Due to the symmetry of the other constraints, this turns out not to be important.
Observe that the constraint $1/\lambda_1+\cdots+1/\lambda_n=n+\Gamma(\hat B)$ implies all feasible points avoid the boundary, $\lambda_j=0$ for some $j$. The set of feasible points is also compact, so the maximizing point for any  continuous objective function will be achieved for a point in the interior of the positive orthant.
This allows us to find optima of smooth functions by the method of Lagrange multipliers. This method uses the following fact: if left-hand sides of the equalities in \eqref{a29} are $F_1(\lambda_1,\ldots,\lambda_n)=\lambda_1+\ldots+\lambda_n$ and $F_2(\lambda_1,\ldots,\lambda_n)=\lambda_1^{-1}+\cdots+\lambda_n^{-1}$, and the objective is $\Phi$, then the optimal assignment satisfies
\eq{\label{a30}
\rank\pare{\bmat{\nabla\Phi(\lambda_1,\ldots,\lambda_n)\\\nabla F_1(\lambda_1,\ldots,\lambda_n)\\\nabla F_2(\lambda_1,\ldots,\lambda_n)}}
=
\rank\pare{\bmat{&\nabla\Phi(\lambda_1,\ldots,\lambda_n)&\\1&\cdots&1\\-1/\lambda_1^2&\cdots&-1/\lambda_n^2}}
\le 2.}
In particular, the determinant of every $3\times 3$ submatrix must vanish.

\begin{lemma}\label{a27}
    \[
    \off(\hat B)
\le\Gamma(B)\cdot\pare{\sqrt{1+{\Gamma(B)}/4}+\sqrt{\Gamma(B)/4}}^2\]
\end{lemma}
\begin{proof}
Put $g=\Gamma(\hat B)=\Gamma(B)$.
We have
\eq{\label{a31}\rmagn{\hat B-I}_F^2=-n+\rmagn{\hat B}_F^2=-n+\lambda_1^2+\cdots+\lambda_n^2.}
We consider maximizing \cref{a31} subject to \eqref{a29}.
This is a smooth objective so by \eqref{a30},
\eq{
\rank\pare{\bmat{
2\lambda_1&\cdots&2\lambda_n\\
1&\cdots&1\\
-1/\lambda_1^2&\cdots&-1/\lambda_n^2
}}\le2
}
By considering a vector $\bmat{1&a&b}$ in the left-nullspace, we see that each eigenvalue must be the root of a common cubic polynomial $z^3+az^2-b=0$.
By Descartes' rule of signs, this polynomial can have at most two positive roots.
Without loss of generality, say 
\[\lambda_1=\cdots=\lambda_k\ge\lambda_{k+1}=\cdots=\lambda_n\]
for some index $k$. Plugging these into these constraints gives
\[
k\lambda_1+(n-k)\lambda_n=n,\quad
\frac k{\lambda_1}+\frac{n-k}{\lambda_n}=n+g.\]
Substituting for $\lambda_n$ and clearing the denominators shows that $\lambda_1$ must be the larger root of
\[z^{2}-\frac{n\left(2k+g\right)}{k\left(n+g\right)}z+\frac{n}{n+g}=0.\]
By applying the quadratic formula, we see that the objective
\[k\lambda_1^2+(n-k)\lambda_n^2\]
is monotonically increasing in $k$, so we should take $k=n-1$. By rearranging the expression obtained from the quadratic formula, this results in
\[\lambda_1=\cdots=\lambda_{n-1}=
\frac{1+\frac{g}{2\left(n-1\right)}-\sqrt{\frac{g}{\left(n-1\right)n}+\frac{g^{2}}{4\left(n-1\right)^{2}}}}{1+{g}/{n}}\qand\lambda_n=n-(n-1)\lambda_1.\]
Then the maximum of the objective function can be expressed as
\spliteq{}{
-n+(n-1)\lambda_1^2+\lambda_n^2
  &=-n+(n-1)\lambda_1^2+\pare{n-(n-1)\lambda_1}^2
\\&=n(n-1)(\lambda_1-1)^2.
}
The resulting expression is monotonically increasing in $n$, so one may upper bound this quantity by taking the limit as $n\to\infty$. This yields
\[\text{off}(\hat B)\le\pare{g/2+\sqrt{g+{g^2}/4}}^2.\]
\end{proof}

\begin{lemma}\label{a28}
\[
1+\Gamma(B)/n
\le\kappa(\hat B)
\le(1+\sqrt{\Gamma(B)/2}+\Gamma(B)/2)^2.\]
\end{lemma}
\begin{proof}
For the lower bound on $\kappa(\hat B)$, note that
\[n\Gamma(\hat B)=n\tr(\hat B^{-1})-n^2=\tr(\hat B)\cdot\tr(\hat B^{-1})-n^2\le n\rmagn{\hat B}\cdot n\rmagn{\hat B^{-1}}-n^2=n^2\kappa(\hat B)-n^2.\]
For the upper bound, put $g=\Gamma(B)=\Gamma(\hat B)$. We have
\eq{\label{a32}\kappa(\hat B)=\max(\lambda_1,\ldots,\lambda_n)/\min(\lambda_1,\ldots,\lambda_n).}
We consider maximizing \cref{a32} subject to \eqref{a29}.
Since the constraints are all symmetric in the variables, we may permute the assignment to the variables which maximizes $\kappa(\hat B)$ so that $\lambda_1$ and $\lambda_n$ are the largest and smallest variables respectively. In particular, this means we may simply maximize the smooth objective
\(\lambda_1/\lambda_n.\)
By \eqref{a30},
\eq{\label{a33}
\rank\pare{\bmat{
1/{\lambda_n}&0&\cdots&0&-\lambda_1/\lambda_n^2\\
1&1&\cdots&1&1\\
-1/\lambda_1^2&-1/\lambda_2^2&\cdots&-1/\lambda_{n-1}^2&-1/\lambda_n^2
}}
\le2.}
The determinant of each $3\times 3$ submatrix must be 0. Take the columns $i,j,n$ for $1<i,j<n$. This determinant is
\[
\frac{\lambda_1}{\lambda_n^2}\pare{\frac1{\lambda_j^2}-\frac1{\lambda_i^2}}=0\]
which implies $\lambda_i=\lambda_j$. The rank condition \cref{a33} thus simplifies dramatically since the middle $n-2$ columns coincide,
\[
\det\pare{\bmat{
1/{\lambda_n}&0&-\lambda_1/\lambda_n^2\\
1&1&1\\
-1/\lambda_1^2&-1/\lambda_2^2&-1/\lambda_n^2
}}
=0.
\]
Many terms cancel in this determinant, and the constraint simplifies to just
\(\lambda_2^2=\lambda_1\lambda_n\).
Divide the constraint $\lambda_1+(n-2)\lambda_2+\lambda_n=n$ by $\lambda_2$ and multiply the constraint $\frac1{\lambda_1}+\frac{n-2}{\lambda_2}+\frac1{\lambda_n}=n+g$ by $\lambda_2$ to obtain
\[
\frac{\lambda_1}{\lambda_2}+(n-2)+\frac{\lambda_n}{\lambda_2}=\frac n{\lambda_2},\quad
\frac{\lambda_2}{\lambda_1}+(n-2)+\frac{\lambda_2}{\lambda_n}=\lambda_2(n+g).\]
The lefthand sides of each of these constraints are both equal to $\sqrt{\lambda_1/\lambda_n}+\sqrt{\lambda_n/\lambda_1}+(n-2)$ since $\lambda_2=\sqrt{\lambda_1\lambda_n}$. Consequently,
\[\lambda_2=\cdots=\lambda_{n-1}=\sqrt{\frac n{n+g}}.\]
Observe that we have deduced the sum $\lambda_1+\lambda_n=n-(n-2)\lambda_2$ and the product $\lambda_1\lambda_n=\lambda_2^2$ so by Viete's formulas $\lambda_1$ and $\lambda_n$ are the larger and smaller roots of
\[z^2-\pare{n-(n-2)\lambda_2}z+\lambda_2^2=0\]
respectively. By applying the quadratic formula and rearranging, we obtain
\[\lambda_1/\lambda_n=1+\frac12\pare{x(x+4)+(x+2)\sqrt{x(x+4)}}\qwhere x=\frac g{\sqrt{1+g/n}+1}.\]
Notice that the expression for $\lambda_1/\lambda_n$ is increasing in $x$, and that $x\le g/2$. Plugging this in gives,
\[\lambda_1/\lambda_n\le
1+\sqrt{2g+g^2(g^{2}+16g+80)/64}+g+{g^{2}}/{8}\]
By treating this as a function of $\sqrt g$ and applying Taylor's theorem, this gives
\spliteq{}{
\kappa(\hat B)
  &\le1+\sqrt{2g}+g+\frac{5}{8\sqrt{2}}g^{1.5}+\frac{g^{2}}{4}
\\&\le(1+\sqrt{g/2}+g/2)^2
}
as required.
\end{proof}

\begin{lemma}\label{a34}
    \[\max\pare{\rmagn{\hat B},\rmagn{\hat B^{-1}}}\le1+\frac{\Gamma(B)}2+\sqrt{\pare{1+\frac{\Gamma(B)}2}^2-1}.\]
\end{lemma}
\begin{proof}
Put $g=\Gamma(\hat B)$. Here we are maximizing either $\lambda_1$ or $1/\lambda_n$ subject to \eqref{a29}. Since the constraints are symmetric, we may equivalently think of the maximum values of $\lambda_1$ and $1/\lambda_1$ instead, which correspond to just the maximum and minimum values of $\lambda_1$. Then by
\eqref{a29}, we have
\[
\det\bmat{1&0&0\\1&1&1\\-1/\lambda_1^2&-1/\lambda_i^2&-1/\lambda_j^2}=0,
\]
i.e. $\lambda_i=\lambda_j$ for all $i,j>1$. By plugging this into the constraints and solving for $\lambda_1$, we see that the extremal values of $\lambda_1$ are the roots of
\[(1+g/n)z^2-(2+g)z+1=0.\]
Notice that the larger root is increasing in $n$ and the smaller root is decreasing in $n$, so we obtain correct bounds by taking the limit as $n$ goes to infinity. The consequence is that the extremal values of $\lambda_1$ lie between the roots of $z^2-(2+g)z+1$. Since the product of the roots is 1, the reciprocal of the smaller root is exactly the larger root, so we finally apply the quadratic formula to obtain
\[
\max(\lambda_1,1/\lambda_1)
\le
1+\frac{\Gamma(B)}2-\sqrt{\pare{1+\frac{\Gamma(B)}2}^2-1}\]
as desired.
\end{proof}

\section{Exact arithmetic}\label{a8}
\label{a35}
In this section we prove $\Gamma(\cdot)$ converges to zero as \Cref{a19} progresses, first by analyzing the change in $\Gamma(\cdot)$ during a single step of the algorithm and then iterating. As a corollary, we show that one- and two-sided versions of \Cref{a19} quickly converge to the factorizations $A=Q\tacc$ and $B=\tacc^*D\tacc$ where $Q$ is nearly orthogonal and $D$ is nearly diagonal. See \Cref{a0} for a precise statement on this convergence. 

\subsection{One-step analysis}
As promised, we derive an analogous formula to \Cref{a23} for how for $\Gamma(\cdot)$ changes upon a single iteration of \Cref{a19}. Interestingly, although we will only apply this formula when $B$ is positive semi-definite, it holds for general matrices and so we state it as such.
\begin{lemma}[Deterministic update formula]
    \label{a36}
    Fix any $B\in\C^{n\times n}$. Let $J\subset[n]$ be any subset of indices and $S\in\C^{\abs J\times\abs J}$ be any nonsingular matrix such that $S^*B_{J,J}S$ is diagonal.
    Denote
    \(T=P_\sigma\bmat{S^{-1}\\&I}P_\sigma^*\) where $\sigma(J)=[k]$.
    Then
    \[\Gamma((T^{-1})^*BT^{-1})=\Gamma(B)+\sum_{i,j\in J,i\neq j}b_{ij}\cdot(B^{-1})_{ji}.\]
\end{lemma}
\begin{proof}
Put $P=P_\sigma$. Note that $\Gamma(\cdot)$ is invariant under conjugation by permutation matrices, so
\[\Gamma(T^{-*}BT^{-1})=\Gamma((P^*T^{-*}P)(P^*BP)(P^*T^{-1}P)).\]
We can express $P^*BP$ and $P^*TP$ as block matrices
\[
P^*BP=\bmat{B_{JJ}&B_{J,J^c}\\B_{J^c,J}&B_{J^c,J^c}}
\qand
P^*TP=\bmat{S^{-1}\\&I}.
\]
For the remainder of the proof, we replace $P^*BP$ with $B$ and $P^*TP$ with $T$. $\odot$ has lower precedence than regular matrix multiplication, so $M_1M_2\odot M_3M_4=(M_1M_2)\odot(M_3M_4)$.
Then
\spliteq{}{
\Gamma(T^{-*}BT^{-1})-\Gamma(B)
=&\tr\pare{T^{-*}BT^{-1}\odot TB^{-1}T^*}-\tr\pare{B\odot B^{-1}}
\\=&\tr\pare{\bmat{ S^*\\&I}B\bmat{ S\\&I}\odot\bmat{ S^{-1}\\&I}B^{-1}\bmat{ S^{-*}\\&I}}-\tr\pare{B\odot B^{-1}}
\\=&\tr\pare{ S^*B_{JJ} S\odot  S^{-1}(B^{-1})_{JJ} S^{-*}}-\tr(B_{JJ}\odot(B^{-1})_{JJ})
}
By assumption, $ S^*B_{J,J} S=D$ is diagonal. In particular, we have
\[
\tr\pare{ S^*B_{JJ} S\odot  S^{-1}(B^{-1})_{JJ} S^{-*}}
=\tr\pare{S^*B_{JJ} S\cdot S^{-1}(B^{-1})_{JJ} S^{-*}}
=\tr\pare{B_{J,J}\cdot(B^{-1})_{J,J}}.\]
This gives
\spliteq{}{
\Gamma(T^*BT)-\Gamma(B)
  &=\tr\pare{B_{JJ}\cdot(B^{-1})_{JJ}}-\tr\pare{B_{JJ}\odot(B^{-1})_{JJ}}
\\&=
\sum_{i,j\in J}b_{ij}\cdot(B^{-1})_{ji}
-
\sum_{j\in J}b_{jj}\cdot(B^{-1})_{jj}
\\&=\sum_{i,j\in J,i\neq j}b_{ij}\cdot(B^{-1})_{ji}
}
as required. 
\end{proof} 

Completing the analogy of \Cref{a23} to \Cref{a24}, we now compute the expected value of the bound from \Cref{a36} under the randomized size $k$ pivot rule. 
\begin{lemma}[Expected update]\label{a37}
Fix any $B\in\C^{n\times n}$.
Let $T$ be as in \Cref{a36} for a uniformly randomly sampled size $k$ pivot set $J\in\binom{[n]}k$. Then
\[\E\Gamma((T^{-1})^*BT^{-1})=\factor\Gamma(B).\]
\end{lemma}
\begin{proof}
\[
\E\pare{\sum_{i,j\in J,i\neq j}b_{ij}\cdot(B^{-1})_{ji}}
=\binom nk^{-1}\sum_{J\in\binom{[n]}k}
\sum_{i,j\in J,i\neq j}b_{ij}\cdot(B^{-1})_{ji}.
\]
Note for each ordered pair $(i,j)$ that the term $b_{ij}\cdot(B^{-1})_{ji}$ appears exactly $\binom{n-2}{k-2}$ times. Thus
\spliteq{}{
\E\pare{\sum_{i,j\in J,i\neq j}b_{ij}\cdot(B^{-1})_{ji}}
&={\binom nk}^{-1}\binom{n-2}{k-2}\sum_{i,j\in [n],i\neq j}b_{ij}\cdot(B^{-1})_{ji}
\\&=-\frac{k(k-1)}{n(n-1)}\Gamma(B).}
\end{proof}

\subsection{Iterated analysis}
\begin{theorem}[Global linear convergence of \Cref{a19}]\label{a38} Let $A\tot,B\tot$ be the sequence of matrices produced by any instantiation of the the one- and two-sided versions \Cref{a19} respectively with randomized size $k$ pivoting. Then
\begin{align*}
    \E\Gamma(B\tot)&=\factor^t\Gamma(B^{(0)}), \\
    \E\Gamma(A\tots A\tot)&=\factor^t\Gamma(A^{(0)*}A^{(0)}).
\end{align*}
    
\end{theorem}
\begin{proof} \Cref{a37} implies
\eq{\label{a39}\E\pare{\Gamma(B\tot)\,|\,B\totm1}=\factor\Gamma(B\totm1).}
By using induction with the law of iterated expectation, we obtain
\spliteq{}{
  \E\Gamma(B\tot)
  &=\E\pare{\E\pare{\Gamma(B\tot)\,|\,B^{(t-1)}}}
\\&=\pare{1-\frac{k(k-1)}{n(n-1)}}\E\Gamma(B^{(t-1)})
\\&=\pare{1-\frac{k(k-1)}{n(n-1)}}^t\Gamma(B^{(0)}).}
\end{proof}

\begin{remark}[Monotonicity and determinism]
    Unlike $\off(\cdot)$, the potential function $\Gamma(\cdot)$ is \textit{not} strictly monotone, it is only monotone in expectation. If one seeks a deterministic pivoting strategy with monotone convergence, one could adopt a greedy strategy which is analogous to \eqref{a17}. Specifically, pick
    \[(i,j)=\argmin_{i,j:i\neq j}\Re\pare{b_{ij}(B^{-1})_{ji}}.\]
    With this pivoting rule, the bounds of \Cref{a38} would hold with an inequality $\le$ instead of equality in expectation.
    We speculate that this could be done efficiently in the setting where one is initially given both $B$ and $B^{-1}$, and updates both $B\tot$ and $(B\tot)^{-1}$ in each iteration.
\end{remark}

\begin{remark}[Kaczmarz-Kac walk]\label{a9}
Consider the special case of the one-sided \Cref{a19}, where $k=2$ and $S$ is either upper or lower triangular. This is precisely the procedure proposed in \cite{b26} (where it is dubbed the ``Kaczmarz-Kac walk'') and concurrently in \cite{b35}. With \Cref{a28} to relate $\Gamma(\cdot)$ to $\kappa(\cdot)$, Theorem \ref{a38} confirms the rate of convergence of this method conjectured by Steinerberger \cite{b26}.
\end{remark}

Our main result is a corollary of \Cref{a38}, showing the convergence of \Cref{a19}.
\begin{corollary}\label{a0}
Let $0<\delta < 1$. Given positive definite $B\in\C^{n\times n}$, run \Cref{a19} with randomized size $k$ pivoting for
\[t\ge \frac{n(n-1)}{k(k-1)}\log\pare{{4n\hat\kappa(B)}/{\delta^2}}\]
iterations. Let $\tacc=\tacc\tot$, $D=B\tot$ be the values returned by the algorithm. Let $D'=D\odot I$ be just the diagonal entries of $D$. Then
\[\magn{B-\tacc^*D'\tacc}\le\magn B\frac{\delta}{\rho-\delta}\]
with probability at least $1-\rho$.
Additionally, $B=\tacc^* D\tacc$ exactly where $D$ is close to diagonal in the sense that
\[\E\sqrt{\sum_{i\neq j}\frac{\abs{d_{ij}}^2}{d_{ii}d_{jj}} }\le\delta\] where $d_{ij}$ is the i$j^{th}$ entry of $D$.
Equivalently, if $B=A^*A$, then the one-sided version returns $Q=A\tot$ forming the exact factorization $A=Q\tacc$ where $Q$ is nearly orthogonal in the sense that
\[\E\sqrt{\sum_{i\neq j}\frac{\abs{\inr{q_i}{q_j}}^2}{\magn{q_i}^2\magn{q_j}^2} }\le\delta\]
where $q_i$ is the $i^{th}$ column of $Q$. Finally, it is possible to execute lines 6-8 so that $\tacc$ is ensured to be (special) unitary or (unit) upper/lower triangular, giving any of the factorizations listed in \autoref{a20}.
\end{corollary}

\begin{proof}%
First note that as outlined in \Cref{a20}, we can apply constraints to $S$ to produce a specific factorization. Note that each listed property of $S$ implies the corresponding property of $T$ in line 7. That is, if $S$ is unitary, $T$ will be unitary, if $S$ is upper triangular, $T$ will be upper triangular, etc. These properties also are closed under multiplication, so the returned transformation $\tacc$ will have the corresponding property as well.
Picking $S$ with the desired constraint is also always possible: these are just the corresponding decompositions of $B_{JJ}$ and $A_J$, which always exist.

We now turn to the quantitative statements in the two-sided case.
\Cref{a37} implies
\[
\E\pare{\Gamma(B\tot)\,|\, B\totm1} = 
\factor \Gamma(B\totm1)\]
By applying the same equality for smaller values of $t$ and using the law of iterated expectation, we have
\spliteq{}{
  \E\Gamma(B\tot)
&=\E\pare{\E\pare{\Gamma(B\tot)\,|\,B^{(t-1)}}}
\\&=\pare{1-\frac{k(k-1)}{n(n-1)}}\E\Gamma(B^{(t-1)})
\\&=\pare{1-\frac{k(k-1)}{n(n-1)}}^t\Gamma(B^{(0)}).}
By \Cref{a38}, our choice of $t$, and \Cref{a28} we have \[\E\Gamma(B\tot)=\factor^t\Gamma(B^{(0)}) \leq \delta^2/4.\] By \Cref{a27} and Jensen's inequality, we have \[\E\sqrt{\sum_{i\neq j}\frac{\abs{d_{ij}}^2}{d_{ii}d_{jj}} } \leq \delta. \] The statements for the one-sided case follow identically.
The first statement in the corollary follows by considering the event
\[
\sqrt{\sum_{i\neq j}\frac{\abs{d_{ij}}^2}{d_{ii}d_{jj}} } \leq n\delta,\]
which occurs with probability $1-\rho$ by Markov's inequality. This is equivalent to
\[\magn{  D'^{-1/2}\tacc^{-*}B\tacc^{-1}D'^{-1/2} - I}_F  \le\delta/\rho\]
which by the triangle inequality and submultiplicativity implies
\spliteq{}{
\magn{B - \tacc^* D'\tacc}_F 
&\le\rmagn{D'^{1/2}\tacc}^2\delta/\rho
\\&=\rmagn{\tacc^* D'\tacc }\delta/\rho
\\&\le(\magn B+\rmagn{B-\tacc^* D'\tacc })\delta/\rho.
}
The desired result follows by rearranging.
\end{proof}

\section{Finite arithmetic analysis}\label{a10}
\newcommand{\ballentrywise}{\ball^{\textnormal{rel}}}
\newcommand{\ballnormwise}{\ball^{\textnormal{nw}}}
\newcommand{\ballstar}{\ball^{(*)}}
\newcommand{\kappaentrywise}{\kappa^{\textnormal{ew}}}
\newcommand{\kappanormwise}{\kappa^{\textnormal{nw}}}
\newcommand{\kappastar}{\kappa^{(*)}}

\newcommand{\nsdist}{\dist_{\textnormal{1}}}
\newcommand{\sdist}{\dist_{\textnormal{2}}}

\newcommand{\xistar}{\xi^{(*)}}
\newcommand{\xistrong}{\xi^{\textnormal{ew}}}
\newcommand{\xiweak}{\xi^{\textnormal{nw}}}
\newcommand{\normentryset}{\set{\textnormal{entrywise},\textnormal{normwise}}}
\newcommand{\strongweakset}{\set{\textnormal{weak},\textnormal{strong}}}
\newcommand{\meps}{{\mathbf u}}
The previous section showed that every instantiation of \Cref{a19} (i.e. every process for selecting $S$ in line 6) with randomized size $k$ pivoting has the same expected behavior in terms of $\Gamma(\cdot)$.
To obtain a convergence result in the presence of round-off error, one needs to assume the updates line 8 and 9 are performed with small mixed forward-backward error (see the diagram \eqref{a40}).

In this setting, $\Gamma(\cdot)$ will play two roles: the first role, as before, is that the convergence of $\Gamma(\cdot)\to0$ tracks the convergence of the algorithm. The second role is in controlling the accumulation of the error acquired during each iteration. 

We let $\wt A\tot$, $\wt B\tot$ denote the state of the one- and two-sided algorithm respectively after $t$ iterations are performed in finite precision with randomized size $k$ pivoting (to contrast with $A\tot$, $B\tot$ without tildes computed exactly). The stability of each iteration can be defined in terms of the following commutative diagram, corresponding to a mixed forward-backward guarantee.
\begin{equation}\label{a40}
\begin{tikzcd}
	{\wt A\tot} & {\wt A\totp1} \\
	{\wt A\totp{1/3}} & {\wt A\totp{2/3}}
	\arrow["\textnormal{floating}", from=1-1, to=1-2]
    	\arrow["\approx"{description}, tail reversed, from=1-1, to=2-1]
	\arrow["\textnormal{exact}", from=2-1, to=2-2]
	\arrow["\approx"{description}, tail reversed, from=1-2, to=2-2]
\end{tikzcd}
\quad\quad\quad\quad\quad
\begin{tikzcd}
	{\wt B\tot} & {\wt B\totp1} \\
	{\wt B\totp{1/3}} & {\wt B\totp{2/3}}
	\arrow["\textnormal{floating}", from=1-1, to=1-2]
    	\arrow["\approx"{description}, tail reversed, from=1-1, to=2-1]
	\arrow["\textnormal{exact}", from=2-1, to=2-2]
	\arrow["\approx"{description}, tail reversed, from=1-2, to=2-2]
\end{tikzcd}
\end{equation}
Later in this section, we will need to refer to the state at various iterations and ``one-third'' iterations. To avoid ambiguity, we always take $t$ to be an integral index (i.e. refers to the state after a complete number of iterations) and $s$ an index in $\frac13\N$ (i.e. can refer to an in-between state).
The diagram commutes in the following sense: if $\wt A\totp1$ is the result of performing one iteration on $\wt A\tot$ in floating point arithmetic, then we are guaranteed the existence of ``intermediate states'' $\wt A\totp{1/3}$ which is close to $\wt A\tot$ and $\wt A\totp{2/3}$ which is close to $\wt A\totp1$ such that $\wt A\totp{2/3}$ is the result of applying one iteration to $\wt A\totp{1/3}$ in exact arithmetic. Analogous states exist for the two-sided variant.
We define ``closeness'' by
\[A\approx A'\iff \nsdist(A,A')\le\eps\qand B\approx B'\iff \sdist(B,B')\le\eps\]
for pseudo-metrics
\[
\nsdist(A,A')=
\sup_j
\magn{\frac{a_j}{ \magn{a_j} }-\frac{a_j'}{ \rmagn{a_j'} }}
\qand
\sdist(B,B')=\sup_{ij}\abs{
\frac{b_{ij}}{\sqrt{b_{ii}b_{jj}}}
-
\frac{b'_{ij}}{\sqrt{b'_{ii}b'_{jj}}}}\]
where $a_j,a'_j$ is the $j$th column of $A,A'$ respectively. Note that we are implicitly assuming here that the diagonal entries of $B$ and $B'$ are positive. Importantly, $\nsdist(A,A')\le\eps$ implies that $\sdist(A^T A,A'^TA')\le2\eps+\eps^2=O(\eps)$ so it suffices to consider only the hypothesis of the two-sided assumption of \eqref{a40}.

\subsection{Control on $\Gamma(\wt B\tot)$}
We begin by proving a finite-arithmetic analog of \Cref{a38}, i.e., we wish to show $\Gamma(\wt B\tot)$ exhibits linear convergence. In the last section, specifically, \eqref{a39}, we showed that the stochastic process $C_{n,k}^{-t}\Gamma(B\tot)$ was a martingale where $C_{n,k}=\factor$. Thus $C^{t}\to0$ roughly implied $\Gamma(B\tot)\to0$ at the same rate.
In this section, we define a new martingale and show $C_{n,k}^{-t}\Gamma(\wt B\tot)$ is (nearly) bounded by it with high probability.
To define the martingale, we introduce the following random variables.\[
X_t=\frac{\Gamma(\wt B\totp{2/3})}{\Gamma(\wt B\totp{1/3})}
\qand Y_{t_1t_2}=\prod_{t=t_1}^{t_2}X_t\]
where we use the convention $Y_{t_1t_2}=1$ if $t_2<t_1$.
\begin{lemma}[Martingale behavior]\label{a41}
For any index $t_1$, the process\[\set{C_{n,k}^{t_1-1-t}Y_{t_1,t}}_{t=t_1}^\infty\] is a nonnegative martingale.
\end{lemma}
\begin{proof}
\Cref{a37} states that
\[\E\pare{X_t\,|\,\wt B\totp{1/3}}=\factor.\] \Cref{a36} shows that $X_t$ is a deterministic function of $\wt B\totp{1/3}$ and the randomly chosen pivot $J\in\binom{[n]}k$.
Since the pivot chosen at each step is independent of previously chosen pivots, one has
\[
\E\pare{X_t\,|\,X_{t-1},\ldots,X_0}
=C_{n,k}
\implies\E(Y_{t_1,t}\,|\,Y_{t_1,t-1},\ldots,Y_{t_1,t_1})=C_{n,k} Y_{t_1,t-1}
.\]
\end{proof}

Our next lemma allows us to control the change in $\Gamma(\cdot)$ between the states $\wt B\totpm{1/3}$ and the state $\wt B\tot$.

\begin{lemma}[Perturbation theory for $\Gamma(\cdot)$]\label{a42}
Let $B$ and $P$ be square matrices with positive diagonal entries. Suppose that $B$ is positive definite and that $\sdist(B,P)\le\eps$ for $\eps<1/(n\Gamma(B)+n^2)$ then $P$ is also positive definite and
\[
\abs{\Gamma(B)-\Gamma(P)}\le n\eps\cdot \frac{(\Gamma(B)+n)^2}{1-n\eps(\Gamma(B)+n)}.
\]
\end{lemma}
\begin{proof}
Let $X=D_B^{-1/2}BD_B^{-1/2}$ and $Y=D_P^{-1/2}PD_P^{-1/2}$.
By definition of $\Gamma(\cdot)$, we have
\[\Gamma(B)=\Gamma(X)=\tr(X^{-1})-n\qand\Gamma(P)=\Gamma(Y)=\tr(Y^{-1})-n.\]
Consequently, we can express their difference as
\spliteq{}{
\abs{\Gamma(X)-\Gamma(Y)}
  =\abs{\tr(X^{-1}-Y^{-1})}
  &=\abs{\tr\pare{X^{-1}\pare{Y-X}Y^{-1}}}
\\&=\abs{\tr\pare{Y^{-1}X^{-1}\pare{Y-X}}}.}
Since the entries of $Y-X$ are bounded by $\eps$, we can apply H\"older's inequality and Cauchy-Schwarz to obtain
\[
\abs{\tr\pare{Y^{-1}X^{-1}\pare{Y-X}}}
\le\eps\sum_{i,j}\abs{(Y^{-1}X^{-1})_{ij}}
\le\eps\magn{Y^{-1}}_F\magn{X^{-1}}_F.\]
Then by rearranging the triangle inequality
\spliteq{}{
\magn{Y^{-1}}_F
  &\le\magn{X^{-1}}_F+\magn{X^{-1}-Y^{-1}}_F
\\&\le\magn{X^{-1}}_F+\magn{X^{-1}}_F\magn{Y^{-1}}_F\magn{X-Y}_F
} one obtains
\[
\magn{Y^{-1}}_F
\le\frac{\magn{X^{-1}}_F}{1 - \magn{Y-X}_F\magn{X^{-1}}_F }
\le\frac{\magn{X^{-1}}_F}{1 - n\eps\magn{X^{-1}}_F }.\]
Finally, note that the above quantity is monotone in $\rmagn{X^{-1}}_F$ and apply
\[
\magn{X^{-1}}_F\le\tr(X^{-1})=\Gamma(X)+n.\]
Note that this argument applies if $Y$ is replaced with any matrix in the line segment $Y_s=sY+(1-s)X$. In particular, $\magn{Y_s^{-1}}_F$ is finite for each choice of $s\in[0,1]$. Since $Y_t$ is a continuous path among Hermitian invertible matrices with one end point, $X$, that is positive definite, we conclude that the other end point, $Y$, is positive definite as well.
\end{proof}

To stitch these two lemmas together, we essentially argue that by selecting $\eps$ sufficiently small, one can ensure that the additive error one obtains from \Cref{a42} between states $\wt B\totm{1/3}$ to $\wt B\totp{1/3}$ does not accumulate too much.

\begin{theorem}\label{a43}
For any time step $\tau$ and parameters $\rho,\delta>0$, if
\[
\eps\le\frac{\delta\rho}{ 2n(\Gamma(B)\tau\rho^{-1}+\delta+n)^2\cdot\tau^2}\]
then
\[
\Pr\pare{\Gamma(\wt B\tot)\le\Gamma(B)\factor^t\rho^{-1}+\delta\quad\forall t<\tau}\ge1-2\rho.\]
\end{theorem}
\begin{proof}
Set $C=1-\frac{n(n-1)}{k(k-1)}$.
Let $y=\tau\rho^{-1}$.
Let $\Omega$ be the event that
\[\Omega\equiv\sup_{0\le t_1,t_2<\tau}C^{t_1-1-t_2}Y_{t_1t_2}\le y\]
and $E_t$ the event that
\[E_t\equiv \Gamma(\wt B\tot)\le\Gamma(B)Y_{0,t-1}+\delta.\]
We claim that $\Omega\subset E_0\cap\cdots\cap E_{\tau-1}$. We argue via induction on $t$. The base case of $\Omega\subset E_0$ follows since $E_0$ is guaranteed. It thus suffices to argue that $\Omega\cap E_0\cap\ldots\cap E_t\subset E_{t+1}$. Condition on the events $\Omega,E_0,\ldots,E_t$.
Define
\[
d_t=\Gamma(\wt B\totp{1/3})-\Gamma(\wt B\totm{1/3})
\qand
d'_t=\Gamma(\wt B\tot)-\Gamma(\wt B\totm{1/3})\]
where we adopt the convention $\wt B^{(-1/3)}=\wt B^{(0)}$
Then
\spliteq{\label{a44}}{
\Gamma(\wt B\totp1)
  &=  \Gamma(B)Y_{0t}+\sum_{j=0}^td_jY_{jt}+d'_{t+1}
\\&\le\Gamma(B)Y_{0t}+(t+1)y\sup_{0\le j\le t}d_j+d'_{t+1}.}
\Cref{a42} shows how to bound $d_i,d'_i$ in terms of $\Gamma(\wt B\tot)$. Specifically,
\eq{\label{a45}
d_j,d'_j
\le 2n\eps\cdot\frac{(\Gamma(\wt B^{(j)}) + n)^2}{1 - n\eps (\Gamma(\wt B^{(j)}) + n)}
\le 2n\eps\cdot\frac{(\Gamma(B)C^jy+\delta + n)^2}{1 - n\eps (\Gamma(B)C^jy+\delta + n)}}
where we use $\Omega,E_j$ in the second inequality. The selection of $\eps$ is such that \eqref{a44} with \eqref{a45} implies $E_{t+1}$ occurs, completing the induction step. The consequence is
\[\Pr(E_{\tau-1}\cap\cdots\cap E_0)\ge\Pr(\Omega).\]
Now consider the event
\[\Omega'\equiv\sup_{t<\tau}C^{-1-t}Y_{0t}\le\rho^{-1}.\]
Then
\[
\Omega'\cap E_t\implies \Gamma(\wt B\tot)\le\Gamma(B) C^t\rho^{-1}+\delta.\]
Therefore by union bound,
\spliteq{\label{a46}}{
\Pr\pare{
\Gamma(\wt B\tot)\le\Gamma(B) C^t\rho^{-1}+\delta\quad\forall t<\tau
}
  &\ge \Pr(E_{\tau-1}\cap\cdots\cap E_0\cap\Omega')
\\&\ge \Pr(E_{\tau-1}\cap\cdots\cap E_0)-\Pr(\neg\Omega')
\\&\ge \Pr(\Omega)-\Pr(\neg\Omega').}
By \Cref{a41}, events $\Omega,\Omega'$ concern the concentration of a nonnegative supermartingale. 
For each $t_1$, one has by Ville's inequality
that
\eq{\label{a47}
\Pr\pare{\sup_{t_2\ge t_1}C^{t_1-1-t_2} Y_{t_1,t_2} \ge a}\le a^{-1}.}
Apply \eqref{a47} for $t_1$ such that $0 \leq t_1 < \tau$ and $a=y$ to obtain by union bound
\[\Pr(\neg\Omega)\le y^{-1}\tau=\rho.\]
Apply \eqref{a47} for $t_1=0$ and $a=\rho^{-1}$ to obtain
\[\Pr(\neg\Omega')\le\rho.\]
Plugging these inequalities into \eqref{a46} gives the final result.
\end{proof}

As we stated at the outset, $\Gamma(\cdot)$ plays two roles: it must converge to $0$ to show the algorithm halts, and it also must never see a large intermediate value to ensure the numerical error does not blow up. To these ends, we have two consequences of \Cref{a43}. In both, we set
\[
\hat B\tot=
\diag\pare{\wt b_{11},\ldots,\wt b_{nn}}^{-1/2}
\wt B\tot
\diag\pare{\wt b_{11},\ldots,\wt b_{nn}}^{-1/2}\]
to be the diagonal-normalization of a matrix.
\begin{corollary}[Boundedness of $\rmagn{(\hat B\tot)^{-1}}$]\label{a11}
    For any time step $\tau$ and parameters $\rho>0$, if
\[\eps\le\frac{\rho^3}{3n^3\rmagn{\hat B^{-1}}^2\tau^4}\]
then
\[
\Pr\pare{
\sup_{t<\tau}\rmagn{(\hat B\tot)^{-1}}
\le\rmagn{\hat B^{-1}}\cdot n\rho^{-1}+3}\ge1-2\rho.\]
\end{corollary}
\begin{proof}
    Apply \Cref{a34} to \Cref{a43}, set $\delta=1$, and use $\factor^t\le1$.
\end{proof}

\begin{corollary}[Approximate convergence]\label{a48}
For any $\rho\in(0,1/2),\delta\in(0,1/4)$, let
\[t\ge \frac{n(n-1)}{k(k-1)}\log\pare{{n\rmagn{\hat B^{-1}}}/{\rho\delta}}.\]
If
\[\eps\le\frac{\delta\rho^3}{3n^3\rmagn{\hat B^{-1}}^2\tau^4}\]
then
\[\max\pare{\rmagn{\hat B\tot},\rmagn{(\hat B\tot)^{-1}}}\le1+2\sqrt\delta\]
with probability $1-2\rho$.
\end{corollary}
\begin{proof}
    Apply \Cref{a34} to \Cref{a43}.
\end{proof}

\subsection{Orthogonalization is stable}
The most general result is that if each iteration of the one-sided algorithm satisfies \eqref{a40}, then the overall algorithm stably computes an orthonormal basis of the column-space of $A=\bmat{a_1&\cdots&a_n}\in\C^{d\times n}$.
This can be extended to a basis of the flag $\spn(a_1), \spn(a_1,a_2), \spn(a_1,a_2,a_3),\cdots.$
We give a handful of examples of implementations of the algorithm which satisfy \eqref{a40}.

We first need to specify a metric on the Grassmanian $\textnormal{Gr}_k(\C^n)$ so that we may quantify the error of the basis output by the algorithm. For subspaces $V$ and $W$ of the same dimension $k$, let $Q_V$ and $Q_W$ denote orthonormal basis for those spaces respectively. Then we use the following standard metric, sometimes referred to as the \textit{gap-metric}, which has a few equivalent expressions \cite{b41},
\[
\dist(V,W)=\magn{Q_VQ_V^*-Q_WQ_W^*}
=\inf_{R}\magn{Q_V-Q_WR}
=\sqrt{1-\sigma_k(Q_V^*Q_W)^2}=\sin\pare{\theta(V,W)},
\]
where $\theta$ is the largest principle angle between $V$ and $W$.

\begin{lemma}[Perturbation theory for subspaces]\label{a49}
Let $\hat A,\hat A'\in\C^{d\times n}$ have unit length columns. Then
\[\dist(\col(\hat A),\col(\hat A'))\le \frac{\sqrt n \nsdist(\hat A,\hat A')}{\max(\sigma_n(\hat A),\sigma_n(\hat A'))}.\]
\end{lemma}
\begin{proof}
Let $\hat A=QR$ be the QR decomposition of $\hat A$. Then
\spliteq{}{
    \dist(\col(A),\col(A'))
    \le
    \magn{Q-\hat A'R^{-1}}
    =
    \magn{(\hat A-\hat A')R^{-1}}
    \le
    \rmagn{\hat A-\hat A'}\magn{R^{-1}}
    =
    \frac{\rmagn{\hat A-\hat A'}}{\sigma_n(\hat A)}.}
Now,
\[
\rmagn{\hat A-\hat A'}
=
\sup_{y}\frac{\rmagn{(\hat A-\hat A')y}}{\magn x\magn y}
\le
\sqrt n\sup_{y}\frac{\rmagn{(\hat A-\hat A')y}}{\magn x\magn y_1}
=\sqrt n\nsdist(\hat A,\hat A').
\]
Finally, the metric is symmetric so we may swap the roles of $A$ and $A'$.
\end{proof}
\begin{theorem}\label{a1}
Say one is given $A\in\C^{d\times n}$. Let \(\hat A=A\diag\pare{\magn{a_1},\ldots,\magn{a_n}}^{-1}\) normalize the columns of $A$. Let $Q=A^{(\tau)}$ be the result of running the one-sided \Cref{a19} with randomized size $k$ pivoting on $A\in\C^{d\times n}$ with full column rank for
\[\tau=O\pare{n^2k^{-2}\log\pare{\frac{n}{\sigma_n(\hat A)\delta}}}\]
iterations. Assume each iteration is computed stably in the sense of \eqref{a40} for
\[
\eps\le\frac{\delta\sigma_n(\hat A)^4}{n^6\tau^4}.
\]
Then with probability $1-2/n$,
\[\magn{Q^*Q-I}\le3\sqrt\delta\qand\dist(\col(Q),\col(A))\le\delta.\]
\end{theorem}
\begin{proof}
\Cref{a48} immediately implies $\magn{Q^*Q-I}\le\delta$.
By the triangle inequality,
\spliteq{}{
\dist(\col(A),\col(\wt A^{(\tau)}))\le\sum_{t=1}^\tau
\biggr(
&\dist(\col(\wt A\totm1),\col(\wt A\totm{2/3}))
\\+
&\dist(\col(\wt A\totm{2/3}),\col(\wt A\totm{1/3}))
\\+
&\dist(\col(\wt A\totm{1/3}),\col(\wt A\tot))
\biggr)
.}
Since $\wt A\totm{1/3}$ is exactly a column operation applied to $\wt A\totm{2/3}$, the middle term in the summand vanishes.
Note that one may replace $\wt A\tos$ with its column-normalization $\hat A\tos$.
The first and last term are bounded by \Cref{a49}:
\[
\dist(\col(A),\col(\wt A^{(\tau)}))\le\tau\sup_{t\in[\tau]}
\pare{
\frac{\sqrt n\nsdist(\hat A\totm1,\hat A\totm{2/3})}{ \max(\sigma_n(\hat A\totm1),\sigma_n(\hat A\totm{2/3})) }
\\+
\frac{\sqrt n\nsdist(\hat A\totm{1/3},\hat A\tot)}{ \max(\sigma_n(\hat A\totm{2/3}),\sigma_n(\hat A\tot)) }
}.
\]
But the distances $\nsdist(\cdot,\cdot)$ in the numerator are bounded by $\eps$ by assumption. Thus we have
\[\dist(\col(A),\col(\wt A^{(\tau)}))\le\frac{2\tau\sqrt n}{ \inf_{s\le\tau} \sigma_n(\hat A\tos) }\cdot\eps.\]
\Cref{a11} shows that
\[\pare{\inf_{s\le\tau} \sigma_n(\hat A\tos)}^{-2}\le\sigma_n(\hat A)^{-2}n^2+3\] with probability $1-2/n$. The selection of $\eps$ gives the desired result.
\end{proof}

\begin{remark}\label{a50}
    The following update rules orthogonalize a pair of unit vectors and can be performed in floating point arithmetic with unit round-off $\meps$ to satisfy \eqref{a40} with $\eps=O(n)\meps$. We describe them for $J=\set{1,2}$ with $\alpha=\inr{a_1}{a_2}$.
\begin{center}
    \begin{tikzpicture}[scale=2,>=Stealth]\label{a51}
\begin{scope}[xshift=0cm]
  \draw[->, line width=1.5pt, black] (0,0) -- (0:1);
  \draw[->, line width=1.5pt, black] (0,0) -- (60:1);
  \draw[->, line width=1.5pt, blue] (0,0) -- (30:1);
  \draw[->, line width=1.5pt, blue] (0,0) -- (30+90:1);
\end{scope}
\begin{scope}[xshift=-2cm]
  \draw[->, line width=2pt, black] (0,0) -- (0:1);
  \draw[->, line width=1.5pt, black] (0,0) -- (60:1);
  \draw[->, line width=1.5pt, red] (0,0) -- (0:1);
  \draw[->, line width=1.5pt, red] (0,0) -- (90:1);
\end{scope}
\begin{scope}[xshift=2cm]
  \draw[->, line width=1.5pt, black] (0,0) -- (0:1);
  \draw[->, line width=1.5pt, black] (0,0) -- (60:1);
  \draw[->, line width=1.5pt, yellow] (0,0) -- (-15:1);
  \draw[->, line width=1.5pt, yellow] (0,0) -- (60+15:1);
\end{scope}
\end{tikzpicture}
\end{center}
The rule depicted the left (red) corresonds to Gram-Schmidt:
\eq{\label{a52}
    \bmat{a_1&a_2}\gets
\bmat{
a_1&
\frac{a_2-\alpha a_1}{\sqrt{1-\alpha^2}}}.
}
The rule depicted in the middle (blue) is a normalized version of taking the SVD:
\eq{\label{a53}
    \bmat{a_1&a_2}\gets
\bmat{
\frac{a_1+a_2}{\sqrt{2+2\alpha}}
&
\frac{a_1-a_2}{\sqrt{2-2\alpha}}}.
}
The rule depicted on the right (yellow) is a rule that treats the inputs symmetrically, and is equivalent to applying the rule depicted in the middle twice in a row:
\eq{\label{a54}
    \bmat{a_1&a_2}\gets
\bmat{
\frac{
(\sqrt{2-2\alpha}+\sqrt{2+2\alpha})a_1
+
(\sqrt{2-2\alpha}-\sqrt{2+2\alpha})a_2
}{\magn{(\sqrt{2-2\alpha}+\sqrt{2+2\alpha})a_1
+
(\sqrt{2-2\alpha}-\sqrt{2+2\alpha})a_2}}
&
\frac{
(\sqrt{2-2\alpha}-\sqrt{2+2\alpha})a_1
+
(\sqrt{2-2\alpha}+\sqrt{2+2\alpha})a_2
}{\magn{(\sqrt{2-2\alpha}-\sqrt{2+2\alpha})a_1
+
(\sqrt{2-2\alpha}+\sqrt{2+2\alpha})a_2}}.
}.
}
\end{remark}

\begin{figure}[h!]
    \centering
\includegraphics[width=0.5\linewidth]{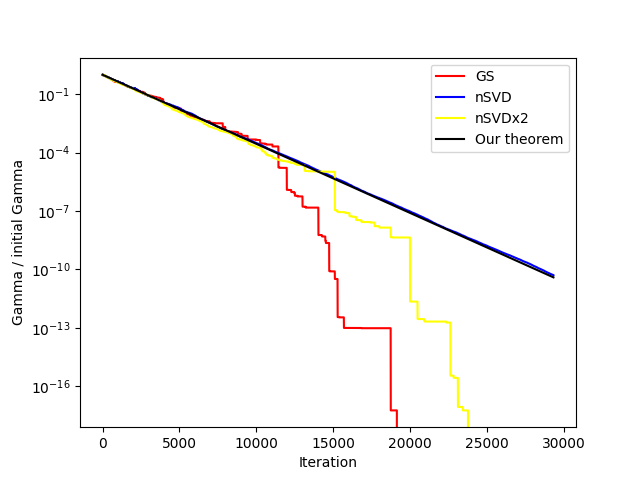}
    \caption{Demonstration of evolution of $\frac{\Gamma((A\tot)^*A\tot)}{\Gamma((A^{(0)})^*A^{(0)})}$ as iterations are performed according to each of the three rules depicted in \Cref{a50}. Rule \cref{a52} is labeled GS in red. Rule \cref{a53} is labeled nSVD in blue. Rule \cref{a54} is labeled nSVDx2 in yellow. The plot for is a single trial applied to a 50x50 matrix with independent Haar unit vector columns. An interesting future line of inquery is understanding why nSVD seems to tightly hug the expected value, whereas GS and nSVD seem to have larger variance or thicker tails leading to a large probability that the observed behavior is better than the expectation.}
    \label{a55}
\end{figure}

\begin{remark}[The QR decomposition]
One can strengthen the result of \Cref{a1} if one implements the update $A_J\gets A_JS$ by performing modified Gram-Schmidt to $A_J$. The returned basis $Q=\bmat{q_1&\cdots&q_n}$ approximates the flag:
    \[\dist(\spn(q_1,\ldots,q_m), \spn(a_1,\ldots,a_m))\le\delta\]
for all $m\le n$.
To see this, note that when using MGS in this way, the ordering of the pivot indices $i_1<\cdots<i_k$ in line 5 ensures that each column is completely unaffected by columns to their right.
Because of this, the algorithm using MGS run on input $\bmat{a_1&\cdots&a_n}$ can be seen as simulating the same algorithm on all the inputs
\[\bmat{a_1}\qand\bmat{a_1&a_2}\qand\bmat{a_1&a_2&\cdots&a_m}\]
with a certain number of no-ops when the pivot $J$ (partially) falls outside the index set.
However, these no-ops exactly counteract the reduced number of iterations required by the smaller input size for convergence: just replace $k$ with the random variable $\abs{J\cap\set{1,\cdots,m}}$. The result would then follow by \Cref{a1}.
\end{remark}

\subsection{Diagonalization and SVD are stable}
Demmel and Veseli{\'c} study the stability of the Jacobi eigenvalue algorithm with an arbitrary pivoting rule. Their results make two assumptions. The first is that the pivoting rule is such that the method converges after some specified number of steps. The second is that the diagonal-normalized condition number of any iterate remains bounded. Under our randomized pivoting rule, both properties are satisfied.  We restate the results of \cite{b0} concerning eigenvalue and singular value error in terms of our notation.

\begin{theorem}[{\cite[Corollary 3.2]{b0}}]\label{a2}
    Given positive definite $B\in\C^{n\times n}$, let $\wt B\tot$ be the $t^{\text{th}}$ iterate of the two-sided \Cref{a19}
    using a Givens rotation in step 6 (i.e. use $k=2$ and follow the Jacobi eigenvalue algorithm) with unit round-off $\meps$.
    Then for each $\tau$,
    \[\sup_{j\in[n]}\frac{\abs{\lambda_j(B)-\wt\lambda_j^{(\tau)}}}{\lambda_j(B)}
    \le O\pare{\sqrt n\tau \meps+n\sdist(\wt B^{(\tau)},I)}\sup_{t\le\tau}\hat\kappa(\wt B\tot)\]
    where $\wt\lambda_j^{(\tau)}$ is the $j$th largest diagonal entry of $\wt B^{(\tau)}$.
\end{theorem}
\begin{theorem}[{\cite[Corollary 4.2]{b0}}]\label{a3}
    Given positive definite $A\in\C^{n\times n}$, let $\wt A\tot$ be the $t^{\text{th}}$ iterate of the one-sided \Cref{a19}
    using a Givens rotation in step 6 (i.e. use $k=2$ and follow the Jacobi eigenvalue algorithm) with unit round-off $\meps$.
    Then for each $\tau$,
    \[\sup_{j\in[n]}\frac{\abs{\sigma_j(B)-\wt\sigma_j^{(\tau)}}}{\sigma_j(B)}
    \le O\pare{ \tau\meps + n^2\meps + n\sdist(A^{(\tau)*}A^{(\tau)},I)}\sup_{t\le\tau}\hat\kappa(\wt A\tots \wt A\tot)\]
    where $\wt\sigma_j^{(\tau)}$ is the length of the $j$th longest column of $\wt A^{(\tau)}$, the infimum is taken over orthogonal matrices, and $\wt B\tot=\wt A\tots \wt A\tot$.
\end{theorem}
Our results allow us to provide a concrete upper bound for randomized size 2 pivoting.
\begin{corollary}\label{a4}When using a randomized size $2$ pivoting rule in the above theorems with $\eps$ as in \Cref{a48}, taking $$\tau=O(n^2\log(n\hat\kappa(B)/\delta))$$ iterations ensures
\[
\sdist(\wt B\tot,I)\le \sqrt\delta
\qquad \text{and} \qquad
\sup_{t\le\tau}\hat\kappa(\wt B\tot)
\le\kappa(\wt B)n^3+3n.\]
with probability $1-O(1/n)$
\end{corollary}
\begin{proof}
    The bound on $\sdist(\wt B\tot,I)$ follows from \Cref{a48}. The bound on $\sup\hat\kappa(\cdot)$ follows from \Cref{a11} along with the observation that $\rmagn{\hat B}\tot\le n$.
\end{proof}

Similar bounds for the eigenvectors and singular vectors can also be found in \cite{b0} (see their Theorems 3.3 and 4.3). We omit repeating the precise statements, but note that the application of \Cref{a4} provides concrete bounds analogous to the bounds for eigenvalues and singular values.%

\section*{Acknowledgments}

We thank Nikhil Srivastava who brought the Jacobi eigenvalue algorithm to our attention when discussing extensions of our earlier work \cite{b35}. We thank Jim Demmel and Ryan Schneider for helpful conversations about the Jacobi eigenvalue algorithm, and the anonymous reviewers for helpful feedback.

\bibliographystyle{alpha}
\bibliography{outbib}

\newcommand{\etalchar}[1]{$^{#1}$}
\begin{thebibliography}{DRVW06}

\bibitem[A{\etalchar{+}}]{b6}
Anderson et~al.
\newblock {\em {LAPACK Users' Guide}}.
\newblock SIAM, 3 edition.

\bibitem[Bis89]{b11}
Christian~H Bischof.
\newblock Computing the singular value decomposition on a distributed system of
  vector processors.
\newblock {\em Parallel Computing}, 11(2):171--186, August 1989.

\bibitem[BKH24]{b37}
Erna Begovi{\'c}~Kova{\u c} and Vjeran Hari.
\newblock Convergence of the complex block jacobi methods under the generalized
  serial pivot strategies.
\newblock {\em Linear Algebra and its Applications}, 699:421–458, October
  2024.

\bibitem[BS89]{b8}
Michael Berry and Ahmed Sameh.
\newblock An overview of parallel algorithms for the singular value and
  symmetric eigenvalue problems.
\newblock {\em Journal of Computational and Applied Mathematics},
  27(1-2):191--213, September 1989.

\bibitem[CETW24]{b30}
Yifan Chen, Ethan~N. Epperly, Joel~A. Tropp, and Robert~J. Webber.
\newblock Randomly pivoted cholesky: Practical approximation of a kernel matrix
  with few entry evaluations.
\newblock {\em Communications on Pure and Applied Mathematics},
  78(5):995--1041, December 2024.

\bibitem[Cha62]{b10}
B.~A. Chartres.
\newblock {Adaptation of the Jacobi Method for a Computer with Magnetic-tape
  Backing Store}.
\newblock {\em The Computer Journal}, 5(1):51--60, January 1962.

\bibitem[CLRT22]{b39}
Erin Carson, Kathryn Lund, Miroslav Rozlo\v{z}n\'ik, and Stephen Thomas.
\newblock Block gram-schmidt algorithms and their stability properties.
\newblock {\em Linear Algebra and its Applications}, 638:150--195, 2022.

\bibitem[Dem97]{b4}
James~W. Demmel.
\newblock {\em Applied Numerical Linear Algebra}.
\newblock Society for Industrial and Applied Mathematics, 1997.

\bibitem[dR89]{b29}
P.~P.~M. de~Rijk.
\newblock {A One-Sided Jacobi Algorithm for Computing the Singular Value
  Decomposition on a Vector Computer}.
\newblock {\em SIAM Journal on Scientific and Statistical Computing},
  10(2):359--371, March 1989.

\bibitem[Drm94]{b21}
Zlatko Drma{\v c}.
\newblock The generalized singular value problem.
\newblock {\em Ph.D. Thesis, Fern Universitat, Hagen, Germany}, 1994.

\bibitem[Drm97]{b25}
Zlatko Drma{\v c}.
\newblock Implementation of jacobi rotations for accurate singular value
  computation in floating point arithmetic.
\newblock {\em SIAM Journal on Scientific Computing}, 18(4):1200–1222, July
  1997.

\bibitem[Drm10]{b36}
Zlatko Drma{\v c}.
\newblock {A Global Convergence Proof for Cyclic Jacobi Methods with Block
  Rotations}.
\newblock {\em SIAM Journal on Matrix Analysis and Applications},
  31(3):1329--1350, 2010.

\bibitem[DRVW06]{b33}
Amit Deshpande, Luis Rademacher, Santosh~S. Vempala, and Grant Wang.
\newblock {\em Theory of Computing}, 2(1):225--247, 2006.

\bibitem[DS24]{b35}
Isabel Detherage and Rikhav Shah.
\newblock {A Kaczmarz-Inspired Method for Orthogonalization}.
\newblock {\em arXiv preprint 2411.16101}, 2024.

\bibitem[DV92]{b0}
James Demmel and Kre\v{s}imir Veseli\'c.
\newblock {Jacobi’s Method is More Accurate than QR}.
\newblock {\em SIAM Journal on Matrix Analysis and Applications},
  13(4):1204--1245, October 1992.

\bibitem[DV06]{b32}
Amit Deshpande and Santosh Vempala.
\newblock {\em Adaptive Sampling and Fast Low-Rank Matrix Approximation}, pages
  292--303.
\newblock Springer Berlin Heidelberg, 2006.

\bibitem[DV08]{b23}
Zlatko Drma{\v c} and Krešimir Veselić.
\newblock New fast and accurate jacobi svd algorithm. i.
\newblock {\em SIAM Journal on Matrix Analysis and Applications},
  29(4):1322–1342, January 2008.

\bibitem[FH60]{b28}
G.~E. Forsythe and P.~Henrici.
\newblock The cyclic jacobi method for computing the principal values of a
  complex matrix.
\newblock {\em Transactions of the American Mathematical Society}, 94(1):1--23,
  1960.

\bibitem[FKN{\etalchar{+}}20]{b19}
Takeshi Fukaya, Ramaseshan Kannan, Yuji Nakatsukasa, Yusaku Yamamoto, and Yuka
  Yanagisawa.
\newblock Shifted cholesky qr for computing the qr factorization of
  ill-conditioned matrices.
\newblock {\em SIAM Journal on Scientific Computing}, 42(1):A477--A503, January
  2020.

\bibitem[FKV04]{b31}
Alan Frieze, Ravi Kannan, and Santosh Vempala.
\newblock Fast monte-carlo algorithms for finding low-rank approximations.
\newblock {\em Journal of the ACM}, 51(6):1025--1041, November 2004.

\bibitem[GMvN59]{b34}
H.~H. Goldstine, F.~J. Murray, and J.~von Neumann.
\newblock The jacobi method for real symmetric matrices.
\newblock {\em Journal of the ACM}, 6(1):59--96, January 1959.

\bibitem[Gol13]{b17}
Gene Golub.
\newblock {\em Matrix Computations}.
\newblock Johns Hopkins University Press, 2013.

\bibitem[Grc11]{b13}
Joseph~F. Grcar.
\newblock How ordinary elimination became gaussian elimination.
\newblock {\em Historia Mathematica}, 38(2):163--218, 2011.

\bibitem[Gre53]{b27}
Robert~T. Gregory.
\newblock Computing eigenvalues and eigenvectors of a symmetric matrix on the
  illiac.
\newblock {\em Mathematics of Computation}, 7(44):215--220, 1953.

\bibitem[GvdV00]{b3}
Gene~H. Golub and Henk~A. van~der Vorst.
\newblock {Eigenvalue computation in the 20th century}.
\newblock {\em Journal of Computational and Applied Mathematics},
  123(1-2):35--65, November 2000.

\bibitem[HB17]{b38}
Vjeran Hari and Erna Begovi{\'c}.
\newblock Convergence of the cyclic and quasi-cyclic block jacobi methods.
\newblock {\em Electron. Trans. Numer. Anal.}, 46:107--147, 2017.

\bibitem[Hes58]{b9}
Magnus~R. Hestenes.
\newblock {Inversion of Matrices by Biorthogonalization and Related Results}.
\newblock {\em Journal of the Society for Industrial and Applied Mathematics},
  6(1):51--90, March 1958.

\bibitem[Hig90]{b16}
Nicholas~J Higham.
\newblock {\em Analysis of the Cholesky decomposition of a semi-definite
  matrix}, pages 161--186.
\newblock Oxford University PressOxford, September 1990.

\bibitem[Hig09]{b15}
Nicholas~J. Higham.
\newblock Cholesky factorization.
\newblock {\em WIREs Computational Statistics}, 1(2):251--254, June 2009.

\bibitem[Jac46]{b2}
C.G.J. Jacobi.
\newblock {Über ein leichtes Verfahren die in der Theorie der
  Säcularstörungen vorkommenden Gleichungen numerisch aufzulösen*).}
\newblock {\em Journal für die reine und angewandte Mathematik},
  1846(30):51--94, 1846.

\bibitem[Kog55]{b1}
E.~G. Kogbetliantz.
\newblock Solution of linear equations by diagonalization of coefficients
  matrix.
\newblock {\em Quarterly of Applied Mathematics}, 13(2):123--132, July 1955.

\bibitem[Luk84]{b24}
Franklin Luk.
\newblock {A Jacobi-like Algorithm for Computing the QR-Decomposition}.
\newblock {\em Cornell University}, 1984.

\bibitem[Mas94]{b12}
Walter~F. Mascarenhas.
\newblock A note on jacobi being more accurate than $qr$.
\newblock {\em SIAM Journal on Matrix Analysis and Applications},
  15(1):215–218, January 1994.

\bibitem[Mat95]{b22}
Roy Mathias.
\newblock Accurate eigensystem computations by jacobi methods.
\newblock {\em SIAM Journal on Matrix Analysis and Applications},
  16(3):977--1003, July 1995.

\bibitem[Par98]{b5}
Beresford~N. Parlett.
\newblock {\em The Symmetric Eigenvalue Problem}.
\newblock Society for Industrial and Applied Mathematics, 1998.

\bibitem[QZL05]{b41}
Li~Qiu, Yanxia Zhang, and Chi-Kwong Li.
\newblock Unitarily invariant metrics on the grassmann space.
\newblock {\em SIAM Journal on Matrix Analysis and Applications},
  27(2):507--531, 2005.

\bibitem[Sch64]{b7}
A.~Sch{\"o}nhage.
\newblock Zur quadratischen konvergenz des jacobi-verfahrens.
\newblock {\em Numerische Mathematik}, 6(1):410--412, December 1964.

\bibitem[Ste25]{b26}
Stefan Steinerberger.
\newblock {Kaczmarz Kac walk}.
\newblock 2025.

\bibitem[TB97]{b14}
Lloyd~N. Trefethen and David Bau, III.
\newblock {\em {Numerical Linear Algebra}}.
\newblock Society for Industrial and Applied Mathematics, Philadelphia, PA,
  1997.

\bibitem[VdS69]{b40}
Abraham Van~der Sluis.
\newblock Condition numbers and equilibration of matrices.
\newblock {\em Numerische Mathematik}, 14(1):14--23, 1969.

\bibitem[VH89]{b20}
K.~Veseli\'{c} and V.~Hari.
\newblock A note on a one-sided jacobi algorithm.
\newblock {\em Numerische Mathematik}, 56(6):627--633, June 1989.

\bibitem[YNYF16]{b18}
Yusaku Yamamoto, Yuji Nakatsukasa, Yuka Yanagisawa, and Takeshi Fukaya.
\newblock Roundoff error analysis of the choleskyqr2 algorithm in an oblique
  inner product.
\newblock {\em JSIAM Letters}, 8(0):5--8, 2016.

\end{thebibliography}

\appendix

\end{document}